\documentclass[smallextended,
    envcountsect,
    envcountsame,
    envcountreset]{svjour3}
    
\usepackage{graphicx}
\usepackage{booktabs}
\usepackage{amsmath}
\usepackage{amstext}
\usepackage{amsfonts}
\usepackage{amssymb}
\usepackage{amsxtra} 
\usepackage{algorithm}
\usepackage{algorithmic}
\usepackage{hyperref}
\usepackage{bbm}
\usepackage{url}
\usepackage[dvipsnames]{xcolor}
\usepackage{tikz}
\usepackage{float} 
\usetikzlibrary{external}
\usepackage{pgfplots}
\usepackage{pgfplotstable}
\pgfplotsset{compat=newest}
\pgfplotsset{plot coordinates/math parser=false}
\newlength\figureheight
\newlength\figurewidth
\usepackage{xargs} 

\newcolumntype{M}[1]{>{\centering\arraybackslash}m{#1}}
\usepackage[square,numbers,sort&compress]{natbib}

\graphicspath{ {./figures/} }

\smartqed

\newcommand \NN {\mathbb{N}}
\newcommand \RR {\mathbb{R}}

\newcommand \Rcal {\mathcal{R}}

\DeclareMathOperator*{\argmin}{argmin}

\DeclareMathOperator{\dist}{dist}

\newcommand{\scp}[2]{\langle #1 \,,\, #2\rangle}

\DeclareMathOperator{\signum}{sign}     

\newcommand{\eqdef}{\overset{\text{def}}{=}} 

\numberwithin{equation}{section}
\graphicspath{ {./figures/} }
\usepackage{subfig}

\date{}
\journalname{}

\begin{document}

\title{Acceleration and restart for the randomized Bregman-Kaczmarz method\thanks{This work has received funding from: the European Union's Framework Programme for Research and Innovation Horizon 2020 (2014-2020) under the Marie Sklodowska-Curie Grant Agreement No. 861137;  UEFISCDI, Romania, PN-III-P4-PCE-2021-0720, under project L2O-MOC, no 70/2022.}}
\titlerunning{Acceleration and restart for  the randomized Bregman-Kaczmarz method}

\author{Lionel~Tondji \and Ion~Necoara \and Dirk~A. Lorenz}
\institute{Lionel Tondji
  \at Institute for Analysis and Algebra, TU Braunschweig, 38092 Braunschweig, Germany.
  \at Center for Industrial Mathematics, University of Bremen, Postfach 330440, 28334 Bremen, Germany,\\
  \email{l.ngoupeyou-tondji@tu-braunschweig.de}
  \and Ion Necoara
  \at Automatic Control and Systems Engineering Department, National University of Science and Technology Politehnica Bucharest, 060042 Bucharest, Romania.
  \at Gheorghe Mihoc-Caius Iacob Institute of Mathematical Statistics and Applied Mathematics of the Romanian Academy, 050711 Bucharest, Romania,\\
  \email{ion.necoara@upb.ro}
  \and Dirk A. Lorenz
  \at Institute for Analysis and Algebra, TU Braunschweig, 38092 Braunschweig, Germany.
  \at Center for Industrial Mathematics, University of Bremen, Postfach 330440, 28334 Bremen, Germany,\\
  \email{d.lorenz@uni-bremen.de}
  }

\maketitle

\begin{abstract}
Optimizing strongly convex functions subject to linear constraints is a fundamental problem with numerous applications. In this work, we propose a block (accelerated) randomized Bregman-Kaczmarz method that only uses a  block of constraints in each iteration to tackle this problem. We consider a dual formulation of this problem in order to deal in an efficient way with the linear constraints. Using convex tools, we show that the corresponding dual function satisfies the Polyak-Lojasiewicz (PL) property, provided that the primal objective function is strongly convex and verifies additionally some other mild assumptions.  However, adapting the existing theory on coordinate descent methods to our dual formulation can only give us sublinear convergence results in the dual space. In order to obtain convergence results in some criterion corresponding to the primal (original) problem, we transfer our algorithm to the primal space, which combined with the PL property allows us to get linear convergence rates. More specifically, we provide a theoretical analysis of the convergence of our proposed method under different assumptions on the objective and demonstrate in the numerical experiments its superior efficiency and speed up compared to existing methods for the same problem.
\end{abstract}



\section{Introduction}
In this work, we consider the fundamental problem of approximating solutions of  large-scale consistent linear systems of the form:
\begin{equation}
\label{eq:LS}
\mathbf{A}x = b
\end{equation}
with matrix $\mathbf{A} \in \RR^{m \times n}$ and right hand side $b \in \RR^m$.  We consider the case that the full matrix is not accessible simultaneously, but that one
can only work with a single or block of rows of the system~\eqref{eq:LS} at a time. Given the possibility of multiple solutions of~\eqref{eq:LS}, we set out to find the unique solution, characterized by the function $f$, i.e.,
\begin{equation}
\label{eq:PB}
\hat f \eqdef \min_{x \in \RR^n} f(x) \quad \text{subject to} \quad  \mathbf{A}x=b,
\end{equation}
with a general strongly convex function $f$. However, we will not assume smoothness (not even differentiability) of $f$. One possible example is $f(x) = \lambda \cdot \|x\|_1 + \frac{1}{2}\|x\|^2_2$ and it is known that this function favors sparse solutions for appropriate choices of $\lambda > 0$, see~\cite{cai2009linearized,LSW14,LWSM14,LS19,yin2010analysis}.  Similarly, by dividing
the components of $x$ into $M$ groups, $x = (x_{(1)},\dots,x_{(M)})$ with $x_{(j)} \in \RR^{n_j}$, the function $f(x) = \lambda \cdot \sum_{j=1}^{M} \|x_{(j)}\|_2 + \frac{1}{2} \cdot \|x\|_2^2$ favours group sparsity. We assume throughout the paper that both $m$ and $n$ are large and that the system is consistent. The $i$-th row of $\mathbf{A}$ is denoted by $a_{i}^{\top}$ and we assume $a_i \neq 0$ for all $i \in [m]:=\{1,\dots,m\}$. Since $f$ is strongly convex, problem~\eqref{eq:PB} has a unique solution $\hat x$. In applications,  it is usually sufficient to find a point which is not too far from $\hat x$. In particular, one chooses the error tolerance $\varepsilon >0$ and aims to find a point $x$ satisfying $\|x - \hat x\|^2_{2} \leq \varepsilon$. Since our method is of a stochastic nature,  the iterates $x$ are random vectors. Hence, our goal will be to compute approximate solutions of~\eqref{eq:PB} that fulfills $\mathbb{E}[\|x - \hat x\|^2_{2}] \leq \varepsilon$, where $\mathbb{E}[ \cdot]$ denotes the expectation w.r.t. randomness of the algorithm.

\subsection{Related work}
\noindent The linear system~\eqref{eq:LS} may be so large that full matrix operations are very expensive or even infeasible. Then, it appears desirable to use iterative algorithms
with low computational cost and storage per iteration that produce good approximate solutions of~\eqref{eq:PB} after relatively few iterations.  The Kaczmarz method and its randomized variants~\cite{gower2019adaptive,gower2015randomized,necoara2019faster,SV09} are used to compute the minimum $\ell_2$-norm solutions of consistent linear systems.  In each iteration $k$, a row vector $a_i^\top$
of $\mathbf{A}$ is chosen at random from the system~\eqref{eq:LS} and the current iterate $x^k$ is projected onto the solution space of that equation to obtain $x^{k+1}$. Note that this update rule requires only low cost per iteration and storage of order $\mathcal{O}(n)$. Recently, a new variant of the randomized Kaczmarz (RK) method namely the randomized sparse Kaczmarz method (RSK)~\cite{LWSM14,LS19} with almost the same low cost and storage requirements has shown good performance in approximating sparse solutions of large consistent linear systems. The papers~\cite{LSW14,LS19} analyze this method by interpreting it as a sequential, randomized Bregman projection method (where the Bregman projection is done with respect to the function $f$), while~\cite{P15} connects it with the coordinate descent method via duality.
Variations of RSK including block/averaging variants can be found in~\cite{du2020randomized,necoara2019faster,needell2014paved,P15,tondji2023faster,tondji2023adaptive}, momentum variants~\cite{lorenz2023minimal}, sampling scheme~\cite{tondji2021linear} and adaptations to least squares problems are given in \cite{Du19,schopfer2022extended,zouzias2013randomized}. In those variants, one usually needs to have access to more than one row of the matrix $\mathbf{A}$ at the cost of increasing the memory. Coordinate descent methods are often considered in the context of minimizing composite convex objective functions \cite{fercoq2015accelerated,necoara2016,nesterov2012efficiency}. At each iteration, they update only one coordinate of the vector of variables using partial derivatives and proximal operators rather than the whole
gradient and the full proximal operator, respectively. A randomized accelerated coordinate descent method was proposed in~\cite{nesterov2012efficiency} for smooth convex functions and later it was extended to composite convex functions in e.g., \cite{fercoq2015accelerated,necoara2016}. Note that accelerated methods transform the coordinate descent method, for which the optimality gap decreases as $\mathcal{O}(1/k)$, into an
algorithm with optimal $\mathcal{O}(1/k^2)$ complexity under mild additional computational cost~\cite{fercoq2015accelerated,nesterov2012efficiency}. {By exploiting the connection between coordinate descent methods and Kaczmarz methods, an accelerated method was proposed in~\cite{10178390}  to solve problem~\eqref{eq:PB} with $\mathcal{O}(1/k^2)$ convergence rate. The present study builds upon the initial findings in~\cite{10178390} and expands upon them with an additional restart algorithm (see Algorithm~\ref{alg:RARBK1}), new theoretical analysis for the primal and dual problem (see Section~\ref{sec:lin_conv}), new linear convergence results for Algorithms~\ref{alg:ARBK} and ~\ref{alg:RARBK1} and more detailed experiments}.

\subsection{Contributions}
\label{sec:contributions}
To the best of our knowledge, blocks accelerated Bregman-Kaczmarz methods for solving linearly constrained optimization problems with strongly convex (possibly nonsmooth) objective functions have not been considered before in the literature. In this work, we provide such methods for solving this type of problems. Hence, this study has the following main contributions:
\begin{enumerate}
    \item Beyond interpreting the Bregman-Kaczmarz method as a dual coordinate descent method, we propose a primal accelerated coordinate gradient algorithm using only blocks of rows of the matrix defining the constraints. We refer to this as the Block Accelerated Randomized Bregman-Kaczmarz method (ARBK) with more details in Algorithm~\ref{alg:ARBK}. We have also extended the method and propose a restart scheme for Algorithm~\ref{alg:ARBK} that exhibits faster convergence.

    \item  Using  tools from convex theory, we show that the corresponding dual function satisfies the Polyak-Lojasiewicz (PL) property, provided that the primal objective function is strongly convex and verifies additionally some other mild assumptions. By exploiting the connection between primal and dual updates, we obtain convergence as a byproduct and convergence rates under different assumptions on the objective function which have not been available so far. In particular, we prove that our restart accelerated method leads to faster (i.e., linear rate) convergence than its standard counterpart. 

    \item We also validate the superior efficiency of our approach  empirically and we provide an implementation of our algorithm in \texttt{Python}.
\end{enumerate}

\subsection{Outline}
The remainder of the paper is organized as follows. Section~\ref{sec:basicnotions} provides notations and a brief overview on convexity and Bregman distances. In Section~\ref{sec:interpretation} we state our method and give an interpretation of it in the dual space. Section~\ref{sec:convergence} provides convergence guarantees for our proposed method. In Section~\ref{sec:numerics}, numerical experiments demonstrate the effectiveness of our method and provide insight regarding its behavior and its hyper-parameters. Finally,
Section~\ref{sec:conclusion} draws some conclusions.


\section{Notation and basic notions}
\label{sec:basicnotions}
For a positive integer $n$ we denote $[n] \eqdef \{1,2,\dots,n\}$. We will assume that the vector $y \in \mathbb{R}^m$ has the following block partition \[y = (y_{(1)}, y_{(2)}, \dots, y_{(M)}),\] composed of $M$ blocks, where $y_{(i)} \in \mathbb{R}^{m_i}, m_i \in \NN$ and $\sum_{i \in [M]} m_i = m$. We used the matrices $\mathbf{U}_i \in \mathbb{R}^{m \times m_i}, i \in [M]$ to partition the identity matrix as $\mathbf{I}_m = (\mathbf{U}_1, \mathbf{U}_2, \dots, \mathbf{U}_{M}) \in \mathbb{R}^{m \times m}$, we get
\begin{equation*}
    y_{(i)} = \mathbf{U}_i^{\top} y \in \mathbb{R}^{m_i}, \;\; \forall y \in \mathbb{R}^{m},  \;\; \forall i \in [M] \quad \text{and} \quad  y = \sum_{i \in [M]} \mathbf{U}_i y_{(i)}. 
\end{equation*}
Moreover, for a given differentiable function $g:\RR^m \to \RR$, the vector of partial derivatives corresponding to the variables in $y_{(i)}$ is $\nabla_{(i)} g(y) = \mathbf{U}^{\top}_i \nabla g(y) \in \mathbb{R}^{m_i}$. A similar partition is applied for $x \in \mathbb{R}^n$, with $\mathbf{V}_i \in \mathbb{R}^{n \times n_i}, n_i \in \NN$ where $\sum_{i \in [M]} n_i = n$. We write $\mathbf{I}_n = (\mathbf{V}_1, \mathbf{V}_2, \dots, \mathbf{V}_{M}) \in \mathbb{R}^{n \times n}$. Hence, we get
\begin{equation*}
    x_{(i)} = \mathbf{V}_i^{\top} x \in \mathbb{R}^{n_i}, \forall x \in \mathbb{R}^{n}, \forall i \in [M] \quad \text{and} \quad  x = \sum_{i \in [M]} \mathbf{V}_i x_{(i)}. 
\end{equation*}
Given a symmetric positive definite matrix $\mathbf{B},$ we denote the induced inner product by 
\begin{equation*}
   \langle x, y \rangle_{\mathbf{B}} \eqdef  \langle x, \mathbf{B} y \rangle = \sum\limits_{i,j \in [n]} x_{(i)}^{\top} \mathbf{B}_{ij} y_{(j)}, \quad x, y \in \RR^{n},
\end{equation*}
and its induced norm by $\|x\|_{\mathbf{B}}^2 \eqdef \langle x, x \rangle_{\mathbf{B}}$. We use the short-hand notation $\|.\|_2$ to mean $\|.\|_{\mathbf{I}_n}$. For an $n\times m$ real matrix $\mathbf{A}$ we denote by $\mathcal{R}(\mathbf{A}), \|\mathbf{A} \|_F, \|\mathbf{A} \|_2$ and $a_i^\top$ its range space, its Frobenius norm, its spectral norm and its $i$-th row, respectively. By $e$ we denote the Euler's number and by $\lceil x \rceil$ we denote the smallest integer greater than or equal to $x$ with $x \in \mathbb{R}$. We use $\textbf{Diag}(d_1, d_2,\dots , d_M)$ to denote the diagonal matrix with $d_1,d_2,\dots, d_M$ on the diagonal. For a random vector $\Theta_j$ that depends on a random index $j \in [n]$ (where $j$ is chosen with probability $p_{j}$) we denote $\mathbb{E} [\Theta_j] \eqdef \sum_{\ell \in [n]} p_\ell \Theta_\ell$ and we will just write $\mathbb{E} [\Theta_j]$ when the probability distribution is clear from the context. Given a vector $x \in \RR^n$, we define the soft shrinkage operator, which acts componentwise on a vector $x$ as
\begin{equation} \label{eq:S}
(S_{\lambda}(x))_j = \max\{|x_j|-\lambda,0\} \cdot \signum(x_j) \quad \forall j \in [n],
\end{equation}
where $x_j$ denotes the $j$-th coordinate of the vector $x$. The indicator function of a set $C$ is denoted by 
\[
\delta_{C}(x) \eqdef 
\begin{cases}
    0, & \text{if } x \in C\\
    +\infty, & \text{if } x \not \in C
  \end{cases} 
\]

\noindent Now, we collect some basic notions on convexity and the Bregman distance, our exposition mainly follows from ~\cite{nesterov2003introductory,Roc:70}.

\subsection{Basic notions}

Let $f:\RR^n \to \RR$ be convex (we assume that $f$ is finite everywhere, hence it is also continuous). The \emph{subdifferential} of $f$ at any $x \in \RR^n$ is defined by
\[
\partial f(x) \eqdef \{x^* \in \RR^n| f(y) \ge f(x) + \langle x^*, y-x \rangle \;\;  \forall y \in \RR^n \},
\]
which is a nonempty, compact and convex set. The function  $f:\RR^n \to \RR$ is said to be \emph{$\sigma$-strongly convex} {with constant $\sigma > 0$}, if for all $x,y \in \RR^n$ and subgradients $x^* \in \partial f(x)$ we have
\[
f(y) \geq f(x) + \langle x^*, y-x \rangle + \tfrac{\sigma}{2} \cdot \|y-x\|_2^2 \,.
\]
If $f$ is $\sigma$-strongly convex, then $f$ is coercive, i.e.
$$
\lim_{\|x\|_2 \to \infty} f(x)=\infty \,.
$$
The {Fenchel conjugate} of a $\sigma$-strongly convex function $f$, denoted $f^{*}:\RR^n \to \RR$,  given by 
\[
f^*(x^*)\eqdef \sup_{y \in \RR^n} \scp{x^*}{y} - f(y),
\]
is also convex, finite everywhere and coercive. Additionally, $f^*$ is differentiable with a \emph{Lipschitz-continuous gradient} with constant $L_{f^*}=\frac{1}{\sigma}$, i.e., for all $x^*,y^* \in \RR^n$ we have
\[
\|\nabla f^*(x^*)-\nabla f^*(y^*)\|_2 \le L_{f^*} \cdot \|x^*-y^*\|_2 \,,
\]
which implies the estimate
\begin{align}  \label{eq:Lip}
f^*(y^*)
\le f^*(x^*)\!  + \scp{\nabla f^*(x^*)}{y^*-x^*} + \tfrac{L_{f^*}}{2}\|x^{*}-y^{*}\|_2^2. 
\end{align}
The function $f^*$ is said to have block-coordinate Lipschitz continuous gradient if there exists $L_{f^*,i}>0$ such that 
\begin{align}
\|\nabla_{(i)} f^*(x^* + V_ih_{(i)}) - \nabla_{(i)} f^*(x^*)\| \leq L_{f^*,i} \cdot \|h_{(i)}\|_2,
\end{align}
for all $x^* \in \RR^n, h_{(i)} \in \RR^{n_i}, i\in [M].$ 

\begin{definition} \label{def:D}
The \emph{Bregman distance} $D_f^{x^*}(x,y)$ between $x,y \in \RR^n$ with respect to a convex function $f$ and a subgradient $x^* \in \partial f(x)$ is defined as
\[
D_f^{x^*}(x,y) \eqdef f(y)-f(x) -\scp{x^*}{y - x}\,.
\]
\end{definition}
Fenchel's equality states that $f(x) + f^*(x^*) = \scp{x}{x^*}$ if $x^*\in\partial f(x)$ and implies that the Bregman distance can be written as
\[
D_f^{x^*}(x,y) = f^*(x^*)-\scp{x^*}{y} + f(y)\,.
\]

\begin{example}[{\cite[Example 2.3]{LS19}}]
\label{exmp:f}
The objective function
\begin{equation} \label{eq:spf}
f(x) \eqdef \lambda \cdot \|x\|_1 + \tfrac{1}{2} \cdot \|x\|_2^{2}
\end{equation}
is strongly convex with constant $\sigma=1$ and its conjugate function can be computed with the soft shrinkage operator from \eqref{eq:S}
\[ 
f^{*}(x^{*}) = \tfrac{1}{2} \cdot \|S_{\lambda}(x^{*})\|_2^{2}, \quad \mbox{with} \quad \nabla f^{*}(x^{*}) = S_{\lambda}(x^{*}) \,.
\]
Its Fenchel conjugate $f^*$ has componentwise Lipschitz gradient with constants $L_{f^*,i} = 1$ and for any $x^*=x+\lambda \cdot s \in \partial f(x)$, with  $s \in \partial (\|x\|_1)$, we have
\[
D_f^{x^*}(x,y)=\frac{1}{2} \cdot \|x-y\|_2^2 + \lambda \cdot(\|y\|_1-\scp{s}{y}) \,.
\]
In particular,  $D_f^{x^*}(x,y)=\frac{1}{2}\|x-y\|_2^2$ for $\lambda = 0.$
\end{example}

\medskip 

\noindent The following inequality is crucial for the convergence analysis of the randomized algorithms. It immediately follows from the definition of the Bregman distance and the assumption of strong convexity of $f$ with constant $\sigma >0$, see~\cite{LSW14}.
For  $x,y \in \RR^n$ and $x^* \in \partial f(x)$ we have
\begin{align} 
\label{eq:D}
D_f^{x^*}(x,y) \geq \frac{\sigma}{2} \|x-y\|_2^2. 
\end{align}

\subsection{Calmness, linear regularity and linear growth conditions}
\label{sec:error_bound_condition}
In this section, we are going to define notions such as calmness, linear regularity, and linear growth ~\cite{schopfer2022extended}, which are necessary conditions for error bounds to hold. Let $B_2$ denote the closed unit ball of the $\ell_2$-norm.

\begin{definition}
\label{ass:calmness}
The (set-valued) subdifferential mapping $\partial f:\RR^n \rightrightarrows \RR^n$ is \emph{calm} at $\hat{x}$ if there are constants $\varepsilon,L>0$ such that
\begin{equation}
\partial f(x) \subset \partial f(\hat{x}) + L \cdot \|x-\hat{x}\|_2 \cdot B_2 \quad \mbox{for any $x$ with} \quad \|x-\hat{x}\|_2\le \varepsilon\,. \label{eq:calm}
\end{equation}
\end{definition}
Note that calmness is a local growth condition akin to Lipschitz-continuity
of a gradient mapping, but for fixed $\hat x$. Of course, any Lipschitz-continuous
gradient mapping is calm everywhere.

\begin{example}
\label{ex:calmness}
    \begin{itemize}
        \item[(a)]  The subdifferential mapping of any convex piecewise linear or quadratic function is calm everywhere. In particular, this holds for the function defined in Eq.~\eqref{eq:spf}.

        \item[(b)] The subdifferential mapping of $f(x) = \lambda \cdot \|x\|_2 + \tfrac{1}{2} \cdot \|x\|_2^{2}$ or of $f(x) = \lambda \cdot \|x\|_1 + \tfrac{1}{2} \cdot \|x\|_2^{2}$ is calm everywhere. For more examples on calmness, we refer the reader to~\cite{schopfer2022extended}.
    \end{itemize}
\end{example}

\begin{definition}
\label{ass:linear_regularity}
Let $\partial f(\hat x) \cap \mathcal{R}(A^{\top}) \neq \emptyset.$ Then, the collection $\{\partial f(\hat x), \mathcal{R}(\mathbf{A}^{\top})\}$ is \emph{linearly regular}, if there is a constant $\nu >0$ such that for all $x^{*} \in \RR^n$ we have
\begin{equation}
\label{eq:linreg}
    \dist(x^{*}, \partial f(\hat x) \cap \mathcal{R}(\mathbf{A}^{\top})) \leq \nu \cdot \bigg(\dist(x^{*}, \partial f(\hat x)) + \dist(x^{*}, \mathcal{R}(\mathbf{A}^{\top}))\bigg).
\end{equation}
\end{definition}
The collection $\{\partial f(\hat x), \mathcal{R}(\mathbf{A}^{\top})\}$ is always linearly regular if, for example, $\partial f(\hat x)$ is polyhedral (which holds e.g., for piecewise linear/quadratic functions $f$).

\begin{definition}
\label{ass:linear_growth}
We say the subdifferential mapping of $f$ \emph{grows at most linearly} if there exist $\rho_1,\rho_2 \ge 0$ such that for all $x \in \RR^n$ and $x^* \in \partial f(x)$ we have
\begin{equation} \label{eq:lineargrowth}
\|x^*\|_2 \le \rho_1 \cdot \|x\|_2 + \rho_2\,.
\end{equation}
\end{definition}
Any Lipschitz-continuous gradient mapping grows at most linearly. Furthermore, the subdifferential mappings of the functions in Example~\ref{ex:calmness} grow at most linearly. Next, we give the definition of the Polyak-Lojasiewicz (PL) inequality which is a lower bound on the size of the gradient as the value of $f$ increases \cite{necoara2019linear}.

\begin{definition} 
Let us consider the class of convex-constrained optimization problems
\[
\textbf{(P):} \quad \hat g = \displaystyle \min_{y \in \mathcal{Y}} g(y),
\]
where $\mathcal{Y} \subseteq \RR^n$ and we denoted by $\mathcal{Y}^* = \displaystyle \argmin_{y \in \mathcal{Y}} g(y)$ the set of optimal solutions of problem (P). We assume that the optimal set $\mathcal{Y}^*$ is nonempty and the optimal value $\hat g$ is finite. Then, a differentiable function $g$ satisfies:
\begin{enumerate}
    \item The \emph{PL inequality} if there exists a constant $\mu > 0$ such that 
\begin{equation}
\label{eq:pl}
    g(y) - \hat g \leq \frac{1}{2\mu} \| \nabla g(y)\|^2_2, \quad \forall y \,\in \mathcal{Y}. 
\end{equation}
\item  The \emph{quadratic growth condition (QG)} if there exists a constant $\mu > 0$ such that 
\begin{equation}
\label{eq:qg}
    g(y) - \hat g \geq \frac{\mu}{2} \dist(y, \mathcal{Y}^*)^2, \quad \forall y \,\in \mathcal{Y}. 
\end{equation}
\end{enumerate}
\end{definition}
\noindent In~\cite[Theorem 2]{karimi2016linear}, it has been shown that for function with Lipschitz-continuous gradient $(PL) \Rightarrow (QG)$. If we further assume that the function is convex then we have $(PL) \equiv (QG)$. Any strongly-convex function satisfies the PL inequality. However, there are many important functions satisfying the PL inequality which are not strongly convex (see e.g., Example \ref{ex:pl_function} and Theorem \ref{th:error_bounds_equality} below).

\begin{example}
\label{ex:pl_function}
Here we give examples of functions satisfying the PL inequality which are not strongly convex.
\begin{enumerate}
    \item For $\lambda >0$, the following function $g$ defined by $g(y) = \tfrac{1}{2} \|S_{\lambda}(A^{\top}y)\|^2_2 - b^{\top}y$ is not strongly convex and satisfies the PL inequality (see Eq.~\ref{eq:pl_constant} for the value of $\mu$).
    \item A  differentiable function $g$ is invex if there exists a vector-valued function $\eta$ such that for any $y$ and $z$, the following inequality holds $g(y) \geq g(z) + \nabla g(z)^{\top} \eta(z,y)$. Note that for invex functions all stationary points are global optima~\cite{karimi2016linear}. As an example in ~\cite{karimi2016linear}, the function  $g(y) = y^2 + 3 \text{sin}^2(y)$ is  invex but non-convex  satisfying PL with $\mu = 1/32$.
\end{enumerate}
\end{example}



\section{Coordinate descent and  Bregman-Kaczmarz methods}
\label{sec:interpretation}
\noindent Note that by a proper scalling of $f$ we can always assume the strong convexity constant $\sigma = 1$. Therefore, in the
sequel we consider 1-strongly convex objective function $f$. To connect the Kaczmarz method to the coordinate descent algorithm, we consider the dual problem to problem \eqref{eq:PB}. Using the conjugate of $f$ we can write the dual objective as the infimum of the Lagrangian $\mathcal{L}(x,y) \eqdef f(x) - y^\top(\mathbf{A}x-b)$ as
\begin{align*}
  \inf_x \mathcal{L}(x,y) &= \inf_x (f(x) - y^\top(\mathbf{A}x-b) ) = b^\top y - f^*(\mathbf{A}^\top y).
\end{align*}
Hence, the dual problem of \eqref{eq:PB} is given by:
\begin{equation}
\label{eq:DP}
\hat \Psi \eqdef \min_{y \in \RR^m}\left[ \Psi(y) := f^*(\mathbf{A}^\top y) - b^\top y\right].
\end{equation}
The primal-dual optimality system
\begin{align}
    \mathbf{A} \hat x = b, \quad
     \hat x = \nabla f^*(\mathbf{A}^\top \hat y),
    \label{eq:opt_sol}
\end{align}
couples the primal and dual optimal solutions $\hat x$ and $\hat y$ respectively. Our objective function in Eq.~\eqref{eq:PB} is strongly convex and the linear system is consistent, therefore there exists a primal solution $\hat x$ and it is unique and strong duality holds in this case. The primal-dual problems can be defined as follows: $\Phi(x) = f(x) + \delta_{\{0\}}(b-\mathbf{A}x)$ and  $ \Psi(y) = -f^*(\mathbf{A}^{\top}y) + b^{\top}y$. Hence, the duality gap is given by:
    \begin{align*}
         \Psi(\hat y) - \Phi(\hat x) = -f^*(\mathbf{A}^{\top} \hat y) + b^{\top} \hat y - f(\hat x) - \delta_{\{0\}}(b-\mathbf{A} \hat x) = - D^{\hat d}_f (\hat x, \hat x) = 0.
    \end{align*}
Therefore, the primal and the dual solutions exist and satisfy equation~\eqref{eq:opt_sol} and thus the optimal solution set $\mathcal{Y}^*$ of~\eqref{eq:DP} is nonempty. Further, the dual objective $\Psi$ is unconstrained, differentiable and its gradient is given by the following expression
\begin{equation}
\label{eq:dual_grad}
    \nabla \Psi(y) = \mathbf{A}\nabla f^*(\mathbf{A}^\top y) - b.
\end{equation}
Moreover, the gradient $\nabla \Psi$ of the dual function is Lipschitz and block-wise Lipschitz continuous w.r.t. the Euclidean norm $\|\cdot\|_2$, with constants \cite{necoara2022linear}:
\begin{equation}\label{eq:Li}
  L_{\Psi}  \eqdef \|\mathbf{A}\|^2_2\ \text{ and }\ L_{\Psi,i}  \eqdef L_i \eqdef \|\mathbf{A}_{(i)}\|_2^2,
\end{equation}
respectively. The following lemma relates the Bregman distance to the optimal primal solution to the distance to optimality of the dual problem.
\begin{lemma}
\label{lm:IterateNormSuboptDual}
Let $b\in\mathcal{R}(\mathbf{A})$ and $(y^k,d^k, x^k)$ be three sequences such that $d^k = \mathbf{A}^\top y^k$ and  $x^k = \nabla f^*(d^k)$. Then,  it holds  
\begin{align}
D_f^{d^k}(x^k,\hat x) = \Psi(y^k)- \hat \Psi.
\end{align}
\end{lemma}
\begin{proof}
By Definition~\ref{def:D} we have that for $d^k = \mathbf{A}^\top y^k$ and $x^k = \nabla f^*(d^k)$. The subgradient inversion formula~\cite{Roc:70} implies that $d^k \in \partial f(x^k)$ and the following is valid:
\begin{align*}
D_f^{d^k}(x^k,\hat x) &= f^*(d^k) + f(\hat x) - \langle d^k,\hat x\rangle \\
&= f^*(\mathbf{A}^{\top}y^k) - \langle b,y^k\rangle  + f(\hat x) \\
&= \Psi(y^k) + \hat f,
\end{align*}
and since  $\hat \Psi = -\max -\Psi = -\hat f$ (by strong duality), the assertion follows.
\end{proof}

\medskip

\noindent The randomized block coordinate descent update applied to minimize  $\Psi$ from \eqref{eq:DP} reads, see  \cite{necoara2016,nesterov2012efficiency}:
\begin{align}
    y^{k+1} &= y^k - \frac{1}{\|\mathbf{A}_{(i)}\|_2^2}  \mathbf{U}_{i} \bigg( \mathbf{A}_{(i)}\nabla f^*(\mathbf{A}^\top y^k) - b_{(i)}\bigg)
    , \label{eq:bcd}
\end{align}
where $i = i_k \in [M]$ is a block row index chosen randomly. For the sampling probabilities of the blocks of rows we will assume the following model which depends on a parameter $\alpha \in [0, 1]$:
 \begin{equation}
   \label{eq:proba}
   p_{\alpha, M}(i) \eqdef p_i = L_i^{\alpha} \cdot \bigg[\sum\limits_{j=1}^M L_j^{\alpha}\bigg]^{-1},
 \end{equation}
 with $L_{i}$ from~\eqref{eq:Li}.
 The case $\alpha=0$ corresponds to uniform probabilities, and the case $\alpha=1$ corresponds to the case where the $i$-th block is sampled with probability proportional to the block Lipschitz constant. Multiplying both sides in~\eqref{eq:bcd} by $\mathbf{A}^\top$, we have
\[
\mathbf{A}^\top y^{k+1} = \mathbf{A}^\top y^{k} - \frac{1}{\|\mathbf{A}_{(i)}\|_2^2}  \mathbf{A}^\top\mathbf{U}_{i} \bigg( \mathbf{A}_{(i)} x^k  - b_{(i)}\bigg).
\]
Using the notations
\begin{align}
\label{eq:primal_iterate}
    d^k = \mathbf{A}^\top y^k, \quad 
    x^k =  \nabla f^*(d^k),
\end{align}
and the fact that $\mathbf{A}_{(i)} = \mathbf{U}_i^{\top} \mathbf{A}$ the block dual coordinate descent iterates transfer in the primal space to block Bregman-Kaczmarz iterates as shown in Algorithm~\ref{alg:BK}. Note that in Algorithm~\ref{alg:BK}, setting $f(x) = \tfrac{1}{2}\|x\|^2_2$ gives us the block randomized Kaczmarz (RK) method \cite{gower2019adaptive,necoara2019faster,SV09}, while $f(x) = \lambda\|x\|_1 + \tfrac{1}{2}\|x\|^2_2$ gives us a block variant of the randomized sparse Kaczmarz (RSK) method \cite{LWSM14,LS19}.
\begin{algorithm}[ht]
  \caption{Block Bregman-Kaczmarz method (BK)}
  \label{alg:BK}
  \begin{algorithmic}[1]
    \STATE {choose  $d^0 =0 \in \RR^n$ and set $x^0 = \nabla f^*(d^{0})$. Choose the number of block $M$} 
    \STATE{\textbf{Output:} (approximate) solution of $\min\limits_{x\in \RR^n} f(x) \quad \text{s.t.} \quad  \mathbf{A}x=b$.}
    \STATE initialize $k=0$
    \REPEAT
    \STATE choose a block row index $i = i_k \in [M]$  randomly
    \STATE update $d^{k+1} = d^k - \frac{1}{\|\mathbf{A}_{(i)}\|_2^2}  \mathbf{A}^{\top}_{(i)} \bigg(\mathbf{A}_{(i)}x^k- b_{(i)}\bigg)$ 
    \STATE update $x^{k+1} = \nabla f^*(d^{k+1})$
    \STATE increment $k =  k+1$
    \UNTIL{a stopping criterion is satisfied}
    \RETURN $x^{k+1}$
  \end{algorithmic}
\end{algorithm}
A more general version of randomized coordinate descent, namely APPROX has been proposed in~\cite{fercoq2015accelerated} to accelerate proximal coordinate descent methods for the minimization of composite functions. The APPROX update applied to solve our dual problem~\eqref{eq:DP} is given by:
\begin{equation}
\label{eq:approx}
\begin{aligned}
    v^k &= (1-\theta_k)y^k + \theta_k z^k \\
    z^{k+1}_{(i)} &= z^k_{(i)} - \frac{1}{\theta_k M \|\mathbf{A}_{(i)}\|_2^2}  \cdot \bigg( \mathbf{A}_{(i)} \nabla f^*(\mathbf{A}^\top v^{k})  - b_{(i)} \bigg)  \\
    y^{k+1} &= v^k + M \theta_k  (z^{k+1} - z^k) \\
    \theta_{k+1} &= \frac{\sqrt{\theta_k^4 + 4\theta_k^2} - \theta_k^2}{2}.
    \end{aligned}
\end{equation}
In the updates~\eqref{eq:approx}, $M$ and $\theta_k$ represent respectively the number of blocks and a sequence used for interpolation. The number of blocks is fixed by the user and the sequence $\theta_k$ is initialised with $\theta_0 = \frac{1}{M}$.  The second equation in~\eqref{eq:approx} can be translated to  full vector operation as $$z^{k+1} = z^k - \frac{1}{\theta_k M \|\mathbf{A}_{(i)}\|_2^2} \mathbf{U}_i  \cdot \bigg( \mathbf{A}_{(i)} \nabla f^*(\mathbf{A}^\top v^{k})  - b_{(i)} \bigg). $$
{We transfer this algorithm to the primal space using the relation $x^{k} = \nabla f^*(\mathbf{A}^{\top}y^{k})$ and this results in a primal block accelerated randomized Bregman Kaczmarz scheme, see Algorithm~\ref{alg:ARBK}.}

\begin{algorithm}[ht]
\vspace{0.1cm}
  \caption{ Block Accelerated Randomized Bregman Kaczmarz method (ARBK)}
  \label{alg:ARBK}
  \begin{algorithmic}[1]
    \STATE{Choose  $z^0 = y^0 = 0 \in \RR^m$. Choose the number of block $M$ and set  $\theta_0 = \frac{1}{M}$, $b\in \RR^m$, $\mathbf{A} \in \RR^{m\times n}$.} 
   \STATE{\textbf{Output:} (approximate) solution $x^{k+1} = \nabla f^*(\mathbf{A}^{\top} y^{k+1})$ of $\min\limits_{x\in \RR^n:\mathbf{A}x=b} f(x)$}
    \STATE initialize $k=0$
    \REPEAT
        \STATE update $v^k = (1 - \theta_k)y^k + \theta_k z^k    $
        \STATE choose a block row index $i = i_k \in [M]$ randomly
    \STATE update $z^{k+1} = z^k - \frac{1}{M\theta_k} \frac{\mathbf{U}_i \cdot \big( \mathbf{A}_{(i)} \nabla f^*(\mathbf{A}^{\top}v^{k})  - b_{(i)} \big)}{\|\mathbf{A}_{(i)}\|_2^2}$
    \STATE update $y^{k+1} = v^{k} + M \theta_k  (z^{k+1} - z^k)$
    \STATE update $\theta_{k+1} = \dfrac{\sqrt{\theta_k^4 + 4\theta_k^2} - \theta_k^2}{2}$
    \STATE increment $k =  k+1$
    \UNTIL{a stopping criterion is satisfied}
    \RETURN $y^{k+1}, x^{k+1} = \nabla f^*(\mathbf{A}^{\top} y^{k+1})$
  \end{algorithmic}
\end{algorithm}

\begin{remark}
We have written the APPROX~\eqref{eq:approx} and ARBK (Algorithm~\ref{alg:ARBK}) in a unified framework to emphasize their similarities. However, the two methods are used for two different purposes. The APPROX method~\eqref{eq:approx} outputs $y^k$, an approximate solution of the dual problem~\eqref{eq:DP}, whereas Algorithm~\ref{alg:ARBK} returns $x^k$, an approximate solution of our primal problem~\eqref{eq:PB}.
\end{remark}
\noindent We  also use the following relation of the sequences, that is, the iterates of APPROX~\eqref{eq:approx} and Algorithm~\ref{alg:ARBK} satisfy for all $k \geq 1$,
\begin{equation*}
    d^{k} = \mathbf{A}^\top y^k, \quad x^k = \nabla f^*(d^k).
\end{equation*}
Note that Algorithm~\ref{alg:BK} can be recovered from Algorithm~\ref{alg:ARBK} by setting
  $\theta_k = \theta_{0}$ for all $k \geq 0$. 

\begin{remark}
    Algorithm~\ref{alg:ARBK} is written in such a way that it would not only be easier for the reader to understand the restarted version, namely Algorithm~\ref{alg:RARBK1} in Section~\ref{sec:convergence}, but also for the convergence analysis. The iterate $y^k$ denotes the dual variables for problem~\ref{eq:DP}. In step 7 in Algorithm~\ref{alg:ARBK}, the iterate $z^{k+1}$ is obtained not using coordinate descent at the current iterate $z^k$, but rather at an extrapolated iterated between $y^k$ and $z^k$, namely $v^k$. The iterate $v^k$ is an additional variable moving from $z^k$ to $y^k$. The sequence $\theta_k$ is a non-increasing positive sequence converging to zero used for interpolation. We would like to specify that it is also possible to fully execute the ARBK method only in the primal space. In this case, the following sequences $z^k, y^k$ and $v^k$ in $\RR^m$ will be replaced by iterates $t^k, d^k$ and $c^k$ in $\RR^n$ respectively:
    \begin{itemize}
        \item In line 1, the initialisation will be replaced by $t^0 = d^0 = 0 \in \RR^n$.
        \item In line 2, the method will return an approximate solution $x^{k+1} = \nabla f^*(d^{k+1})$.
        \item In line 5, update $c^k = (1-\theta_k)d^k + \theta_k t^k$
        \item In line 7, update $t^{k+1} = t^k - \frac{1}{M\theta_k} \frac{\mathbf{U}_i \cdot \big( \mathbf{A}_{(i)} \nabla f^*(c^{k})  - b_{(i)} \big)}{\|\mathbf{A}_{(i)}\|_2^2}$
    \item In line 8, update $d^{k+1} = c^{k} + M \theta_k  (t^{k+1} - t^k)$
    \item In line 12, return $x^{k+1} = \nabla f^*(d^{k+1})$.
    \end{itemize}
With this primal version, we only compute the iterate $x^k$ at the end of the iterations. In Algorithm~\ref{alg:ARBK} we do not need to compute $x^0$ because we only need the iterate $y^0$. 
\end{remark}

\section{Convergence Analysis}
\label{sec:convergence}
In this section we first review a basic convergence result for APPROX from~\eqref{eq:approx} that is valid in the dual space, and which will be used later to
build convergence results for the proposed method Algorithm~\ref{alg:ARBK} in the primal space. We first recall the following properties on the sequence $\{\theta_k\}_{k \ge 0}$.
\begin{lemma}[{\cite[Note 2]{tseng2008accelerated}}]
\label{lm:sequence}
The sequence $\theta_k$ defined by $\theta_0 \leq 1$ and  $\theta_{k+1} = \frac{\sqrt{\theta_k^4 + 4\theta_k^2} - \theta_k^2}{2}$ satisfies
\begin{equation}
\begin{aligned}
   \frac{(2-\theta_0)}{k + (2-\theta_0)/\theta_0} &\leq \theta_k \leq \frac{2}{k + 2/\theta_0}, \\ 
   \frac{1-\theta_{k+1}}{\theta^2_{k+1}} &= \frac{1}{\theta^2_k}, \quad \forall k=0,1,\dots \\  
   \theta_{k+1} &\leq \theta_k, \quad \forall k=0,1,\dots
\end{aligned}
\end{equation}
\end{lemma}
\medskip

\noindent The next lemma gives a recurrence relation for the dual sequences generated by Algorithm~\ref{alg:ARBK}.
\begin{lemma}[{\cite[see Eq. (52)]{fercoq2015accelerated}}]
\label{lm:fercoq}
Consider the linearly constrained optimization problem~\eqref{eq:PB} with $1$-strongly convex $f:\RR^n \to \RR$.  Let $y^k, z^k$ be the sequences generated by Algorithm~\ref{alg:ARBK}, $\theta_{0} = \frac{1}{M}$ and any solution $\hat y \in \mathcal{Y}^*$. Then, it holds:
\begin{align}
\frac{1}{\theta^2_{k-1}}\mathbb{E} [\Psi(y^{k}) - \hat \Psi] + \frac{1}{2\theta^2_0} \mathbb{E}[\|z^k - \hat y\|^2_{\mathbf{B}}]   \leq \frac{1-\theta_0}{\theta^2_0} (\Psi(y^{0}) - \hat \Psi) + \frac{1}{2\theta^2_0}\|y^0 - \hat y\|^2_{\mathbf{B}},
\end{align}
with ${\mathbf{B}} = \textbf{Diag}(\|\mathbf{A}_{(1)}\|_2^2, \|\mathbf{A}_{(2)}\|_2^2, \dots, \|\mathbf{A}_{(M)}\|_2^2)$.
\end{lemma}
For our proposed methods we are interested in showing convergence results for the primal iterates $x^k$ given in~\eqref{eq:primal_iterate} and not for the dual iterates $y^k$. By showing that the rewriting of the iterations of the APPROX method in the primal space yields ARBK method, and using recent convergence results on coordinate descent methods for problems with non-strongly convex objective functions \cite{fercoq2015accelerated,necoara2016,10178390} and Lemma~\ref{lm:IterateNormSuboptDual}, we obtain the following convergence results for ARBK (and consequently also for BK) as a byproduct. We denote  (similarly to~\cite{nesterov2012efficiency}) the size of the initial level set in the standard Euclidean norm by
\begin{equation}
    R_{\alpha}(y^0) = \text{max}_{y} \big\{\text{min}_{\hat y  \in \mathcal{Y}^{*}} \|y - \hat y\|_{\alpha}: \Psi(y) \leq \Psi(y^{0})\big\},
\end{equation}
where
\begin{equation}
    \|y\|_{\alpha}^2 \eqdef \sum\limits_{i=1}^{M} L_i^{\alpha} \|y_{(i)}\|^2_2, \quad \quad \alpha \in [0,1]
\end{equation}
The previous constant $R_{\alpha}(y^0)$, which measures the size of the initial level set,  appears in our convergence rate and it is common in the coordinate descent literature. 
Moreover, we also denote
\begin{equation}
\bar{L}_{\alpha} \eqdef \frac{1}{M} \sum\limits_{j=1}^M L_j^{\alpha}\  \text{ and }\ \bar L \eqdef \bar{L}_{1}\label{eq:barL}
\end{equation}
\subsection{Sublinear convergence results}
In this section, we are going to show sublinear convergence results for both Algorithm~\ref{alg:BK} and~\ref{alg:ARBK}.

\begin{theorem}
\label{thm:1}
Consider the linearly constrained optimization problem~\eqref{eq:PB} with $1$-strongly convex $f:\RR^n \to \RR$.  Let $x^k, d^k$ be the sequences generated by BK i.e. Algorithm~\ref{alg:BK}. At every iteration $k$ of the BK method, a block index $i = i_k  \in [M]$ is sampled with probability $p_{\alpha, M}(i)$ from Eq.~\eqref{eq:proba} and $\alpha\in [0,1]$.
Then, it holds that
\begin{align}
\label{eq:sub_conv}
\frac{1}{2}\mathbb{E}\big[\|x^k -\hat{x}\|^2_2\big] \leq\mathbb{E}\big[D_f^{d^k}(x^k,\hat{x})\big] \leq \frac{2M\bar{L}_{\alpha}}{k+4}  R^{2}_{1-\alpha}(y^0).
\end{align}
\end{theorem}
\begin{proof}
    Let $y^k$ be the sequence generated by iteration~\eqref{eq:bcd}. It holds from~\cite[Theorem 1] {nesterov2012efficiency}:
    \begin{equation*}
        \mathbb{E}\bigg[\Psi(y^k) - \hat \Psi\bigg] \leq \frac{2}{k+4} \cdot \bigg[\sum\limits_{j=1}^M L_j^{\alpha}\bigg] \cdot R^{2}_{1-\alpha}(y^0)  \quad \forall k \geq 0. 
    \end{equation*}
Furthermore, from Lemma~\ref{lm:IterateNormSuboptDual}, it holds $D_f^{d^k}(x^k,\hat x) = \Psi(y^k)- \hat \Psi$ which give us the upper bound part. The lower bound part follows from the fact that $f$ is $1$-strongly convex, hence satisfying Eq.~\eqref{eq:D}.
\end{proof}

\begin{remark}
    For the two extremal variants of the choice of probabilities in~\eqref{eq:proba} we get in~\eqref{eq:sub_conv} the following results:
    \begin{enumerate}
    \item Uniform probabilities (i.e., $\alpha = 0$):
    \begin{align*}
        \mathbb{E}\big[\|x^k -\hat{x}\|^2_2\big] \leq \frac{4M}{k+4}  R^{2}_{1}(y^0).
    \end{align*}
    \item Probabilities proportional to $L_i$ (i.e., $\alpha = 1$):
    \begin{align*}
         \mathbb{E}\big[\|x^k -\hat{x}\|^2_2\big]  \leq \frac{4M\bar L}{k+4} R^{2}_{0}(y^0).
    \end{align*}
\end{enumerate}
\end{remark}
The next theorem shows convergence results for our accelerated method,  Algorithm~\ref{alg:ARBK}.
\begin{theorem}
\label{thm:2}
Consider the linearly constrained optimization problem~\eqref{eq:PB} with $1$-strongly convex $f:\RR^n \to \RR$.  Let $x^k, d^k$ be the sequences generated by ARBK (Algorithm~\ref{alg:ARBK}) with uniform probabilities, $\theta_0 = \tfrac{1}{M}$ and any $\hat y \in \mathcal{Y}^*$. Then, it holds:

\begin{align}
\frac{1}{2}\mathbb{E}\big[\|x^k -\hat{x}\|^2_2\big] \leq \mathbb{E} \big[D_f^{d^k}(x^k,\hat{x})\big] \leq  \frac{4M^2}{(k-1+2M)^2} C_0,
\end{align}
where
\[C_0 = \bigg(1 - \frac{1}{M}\bigg) D_f^{d^0}(x^0,\hat{x}) + \frac{1}{2}\|y^0 - \hat y\|^2_{\mathbf{B}}\]
with ${\mathbf{B}} = \textbf{Diag}(\|\mathbf{A}_{(1)}\|_2^2, \|\mathbf{A}_{(2)}\|_2^2, \dots, \|\mathbf{A}_{(M)}\|_2^2)$.
\end{theorem}
\medskip

\begin{proof}
From Lemma~\ref{lm:fercoq} and Lemma~\ref{lm:IterateNormSuboptDual}, it holds:
\begin{align*}
\frac{1}{\theta^2_{k-1}}\mathbb{E} \big[D_f^{d^k}(x^k,\hat{x})\big]   &\leq \frac{1-\theta_0}{\theta^2_0} D_f^{d^0}(x^0,\hat{x}) + \frac{1}{2\theta^2_0}\|y^0 - \hat y\|^2_{\mathbf{B}}, \\
\mathbb{E} \big[D_f^{d^k}(x^k,\hat{x})\big]   &\leq \frac{\theta^2_{k-1}}{\theta^2_0} \bigg( (1-\theta_0) D_f^{d^0}(x^0,\hat{x}) + \frac{1}{2}\|y^0 - \hat y\|^2_{\mathbf{B}} \bigg), \\
\mathbb{E} \big[D_f^{d^k}(x^k,\hat{x})\big]   &\leq \frac{4M^2}{(k-1+2M)^2} \bigg( (1-\theta_0) D_f^{d^0}(x^0,\hat{x}) + \frac{1}{2}\|y^0 - \hat y\|^2_{\mathbf{B}} \bigg),
\end{align*}
where we used the first inequality in Lemma~\ref{lm:sequence}, i.e., $\frac{\theta^2_{k-1}}{\theta^2_0} \leq \frac{4}{(k-1+2M)^2} \cdot \frac{1}{\frac{1}{M^2}} = \frac{4M^2}{(k-1+2M)^2}$. The first inequality in Theorem~\ref{thm:2} follows from Eq.~\eqref{eq:D}.
\end{proof}
\noindent  Theorem~\ref{thm:2} shows that the iterates $x^k$ of Algorithm~\ref{alg:ARBK} (ARBK) converge at the rate $\mathcal{O}(1/k^2)$, thus accelerating its standard counterpart, Algorithm~\ref{alg:BK} (BK), which has $\mathcal{O}(1/k)$ rate of convergence (cf. Theorem~\ref{thm:1}). To the best of our knowledge, accelerated Kaczmarz variants have not yet been proposed for problem \eqref{eq:PB} and the convergence guarantees for ARBK algorithm from Theorem~\ref{thm:2} are new.  Moreover, the convergence rates from both theorems (Theorem~\ref{thm:1} and Theorem~\ref{thm:2}) depend on the number of blocks, $M$, proving that it is beneficial to consider blocks of the matrix $\mathbf{A}$ in our previous algorithms. In fact, one can notice that using fewer blocks (i.e., small $M$) leads to better bounds on the convergence rate. However, small $M$ implies large block sizes which increases the computational time of the subproblem. From our computational results from Section~\ref{sec:numerics} one can see that small $M$ indeed yields less number of full iterations (epochs). Under some regularity assumptions on the objective function $f$, see Section~\ref{sec:basicnotions},  we can even show linear convergence for the BK and ARBK methods, respectively.

\subsection{Linear convergence results}
\label{sec:lin_conv}
Consider the feasible, convex and linearly constrained optimization problem~\eqref{eq:PB}. To obtain convergence rates for the solution algorithm, we will estimate the Bregman distance of the iterates to the solution $\hat x$ by error bounds of the form $D_f^{x^*}(x,\hat x) \leq \theta \cdot \|\mathbf{A}x - b\|^2_2$. We know that such error bounds always hold if $f$ has a Lipschitz-continuous gradient \cite{necoara2022linear}. But they also hold under weaker conditions. For example, the following inequality holds for problem~\eqref{eq:PB} with objective function defined by~\eqref{eq:spf}:
\begin{equation*}
    D_f^{x^*}(x,\hat x) \leq \theta(\hat x) \cdot \|\mathbf{A}x - b\|^2_2,
\end{equation*}
we refer the reader to~\cite{schopfer2022extended} for more details on the constant $\theta(\hat x)$.

\begin{example} \label{exmp:EBsparse}
Let $\hat{x}$ be the unique solution of~\eqref{eq:PB} with objective function $f(x) = \lambda \cdot \|x\|_{1} + \tfrac{1}{2} \cdot \|x\|_{2}^{2}$.
Then there exists $\gamma(\hat{x}) >0$ such that for all $x \in \RR^n$ and $x^* \in \partial f(x) \cap \Rcal(\mathbf{A}^{\top})$ we have
\[
D_f^{x^*}(x,\hat{x})\le \gamma(\hat{x}) \cdot \|\mathbf{A}x-b\|_{2}^2\,.
\]
Let $\mathbf{A}_J \neq 0$ denotes a submatrix that is formed by columns of $\mathbf{A}$ indexed by $J\subseteq [n]$ and let $\sigma^{+}_{\min}(\mathbf{A}_J)$ denote its smallest positive singular value. We set
\begin{equation*} \label{eq:sing_value}
    \tilde{\sigma}_{\min}(\mathbf{A}) \eqdef \min\{\sigma^{+}_{\min}(\mathbf{A}_J) \mid J\subseteq [n], \mathbf{A}_J\neq 0 \},
\end{equation*}
and for $x \neq 0$ we define $| x|_{\mathrm{min}} \eqdef \min\{|x_j| \mid x_j\neq 0\}$. 
Based on the results of~\cite{LY13}, an explicit expression of $\gamma(\hat{x})$ for $\hat{x} \not= 0$ was given in~\cite{LS19} as follows:
\begin{equation}
\label{eq:pl_constant}
\gamma(\hat{x})=\frac{1}{\tilde{\sigma}_{\min}^2(\mathbf{A})} \cdot \frac{|\hat{x}|_{\min} + 2 \lambda}{|\hat{x}|_{\min}} \,.
\end{equation}
\end{example}

\noindent To clarify the assumptions under which such error bounds hold for more general objective functions, notions such as calmness, linear regularity, and linear growth ~\cite{schopfer2022extended} have been defined in Section~\ref{sec:error_bound_condition}. The next result shows that the dual function satisfies some PL inequality.

\begin{theorem}[{\cite[Theorem 3.9]{schopfer2022extended}}]
\label{th:error_bounds_equality}
Consider the linearly constrained optimization problem~\eqref{eq:PB} with strongly convex objective function $f:\RR^n \to \RR$.  Assume that $f$ is calm at the unique solution $\hat x$ of \eqref{eq:PB}, that $\partial f$ grows at most linearly (cf. Def.~\ref{ass:linear_growth}, and Def.~\ref{ass:calmness}, respectively), and the collection $\{\partial f(\hat x), \mathcal{R}(\mathbf{A}^{\top})\}$ is linearly regular (cf. Def.~\ref{ass:linear_regularity}). Then, $\Psi$ satisfies the PL inequality~\eqref{eq:pl}, i.e., there exists $\gamma(\hat x) >0$ such that for all $x \in \RR^n$ and $d \in \partial f(x) \cap \mathcal{R}(\mathbf{A}^{\top})$ we have the global error bound:
\begin{equation}
    \label{eq:EB}
   D_f^{d}(x,\hat x) = \Psi(y)-\hat \Psi \leq \gamma(\hat x) \cdot \|\nabla \Psi(y)\|^2_2 = \gamma(\hat x) \cdot \|\mathbf{A}x-b\|_2^2.
\end{equation}
\end{theorem}

\noindent The following theorem gives a convergence result for the iterates generated by Algorithm~\ref{alg:BK}.

\begin{theorem}
\label{th:linear_convergence_bk}
Consider the linearly constrained optimization problem~\eqref{eq:PB} with $1$-strongly convex $f:\RR^n \to \RR$. Assume that $f$ is calm at the unique solution $\hat x$ of \eqref{eq:PB} and that $\partial f$ grows at most linearly (cf. Def.~\ref{ass:linear_growth}, and Def.~\ref{ass:calmness}, respectively), and the collection $\{\partial f(\hat x), \mathcal{R}(\mathbf{A}^{\top})\}$ is linearly regular (cf. Def.~\ref{ass:linear_regularity}). Let $x^k$ be the sequence generated by BK (Algorithm~\ref{alg:BK}) with probability
 \begin{equation*}
     p_{\alpha, M}(i) = p_i = L_i^{\alpha} \cdot \bigg[\sum\limits_{j=1}^M L_j^{\alpha}\bigg]^{-1}, \;\; \text{where} \;\; \alpha \in [0, 1].
 \end{equation*}
 Then, it holds
 \begin{align}
\label{eq:bd_conv}
   \mathbb{E}[D_f^{d^{k+1}}(x^{k+1},\hat{x})] \leq \bigg(1 - \frac{1 }{2M\gamma(\hat x)\bar L_{\alpha} \bar{L}^{1-\alpha}}\bigg) \mathbb{E}[D_f^{d^{k}}(x^{k},\hat{x})],
\end{align}
and
\begin{align}
\label{eq:it_conv}
   \mathbb{E}[\|x^{k} -\hat{x}\|_2^2] \leq 2\bigg(1 - \frac{1 }{2M \gamma(\hat x) \bar L_{\alpha} \bar{L}^{1-\alpha}}\bigg)^k D_f^{d^{0}}(x^{0},\hat{x})
\end{align}
\end{theorem}

\begin{proof}
Let $y^k$ be the sequence generated by iteration~\eqref{eq:bcd} and consider Eq.~\eqref{eq:dual_grad}. The following relationship holds \[y^{k+1}_{(i)} -y^k_{(i)} = -\nabla \Psi_{(i)}(y^k).\] From the fact that the
gradient of the function $\Psi$ is blockwise Lipschitz continuous with constants $L_i = \|\mathbf{A}_{(i)}\|_2^2$, we have:
\begin{align*}
    \Psi(y^{k+1}) \leq \Psi(y^{k}) - \frac{1}{2\|\mathbf{A}_{(i)}\|_2^2} \|\nabla_{(i)} \Psi (y^k)\|^2,
\end{align*}
and thanks to Theorem~\ref{th:error_bounds_equality} we have the following relation
\begin{equation*}
    \Psi(y)-\min \Psi \leq \gamma(\hat x) \cdot \|\nabla \Psi(y)\|^2_2.
\end{equation*}
Then with $p_{M}(i) = p_i = L_i^{\alpha} \cdot \big[\sum\limits_{j=1}^M L_j^{\alpha}\big]^{-1}, \;\; \text{where} \;\; \alpha \in [0, 1]$, 
we have:
\begin{align*}
   \mathbb{E}_{k} \bigg[ \Psi(y^{k+1}) - \hat \Psi\bigg] &\leq \Psi(y^{k})- \hat \Psi - \sum\limits_{i=1}^{M}  \frac{p_i }{2L_{i}} \|\nabla_{(i)} \Psi (y^k)\|^2_2 \\
   &= \Psi(y^{k}) - \hat \Psi - \frac{1 }{2\sum\limits_{j=1}^M L_j^{\alpha}}\sum\limits_{i=1}^{M}  \frac{1 }{L_{i}^{1-\alpha}} \|\nabla_{(i)} \Psi (y^k)\|^2_2 \\
   &\leq \Psi(y^{k}) - \hat \Psi - \frac{1}{2M\bar L_{\alpha} \bar{L}^{1-\alpha}} \cdot \|\nabla \Psi (y^k)\|^2_2, \\
&\leq \bigg(1 - \frac{1 }{2M\gamma(\hat x)\bar L_{\alpha} \bar{L}^{1-\alpha}}\bigg) (\Psi(y^{k})  - \hat \Psi),
\end{align*}
where $\mathbb{E}_{k}$ denotes the conditional expectation on all indices sampled up to iteration $k$ and we used the relations   $M \bar L_{\alpha} = \sum_{j \in [M]} L^{\alpha}_{j}$ and  $\bar L \geq L_i, \, \forall i \in [M]$, with $\bar L = \bar L_{1}$. Then, using the rule of total expectation gives us:
\begin{align*}
     \mathbb{E} \big[ \Psi(y^{k+1})  - \hat \Psi \big] \leq \bigg(1 - \frac{1 }{2M\gamma(\hat x)\bar L_{\alpha} \bar{L}^{1-\alpha}}\bigg) \cdot \mathbb{E} \big[ \Psi(y^{k})    - \hat \Psi\big].
\end{align*}
Furthermore, from Lemma~\ref{lm:IterateNormSuboptDual}, it holds $D_f^{d^k}(x^k,\hat x) = \Psi(y^k)- \hat \Psi$, which gives Eq.~\eqref{eq:bd_conv}. 
Unrolling Eq.~\eqref{eq:bd_conv} give us :
\begin{align*}
    \mathbb{E}[D_f^{d^{k}}(x^{k},\hat{x})] \leq \bigg(1 - \frac{1 }{2M\gamma(\hat x)\bar L_{\alpha} \bar{L}^{1-\alpha}}\bigg)^k \cdot D_f^{d^{0}}(x^{0},\hat{x}).
\end{align*}
The second estimate  follows from the fact that $f$ is $1-$strongly convex, hence satisfies $\frac{1}{2} \|x^k - \hat x\|^2_2 \leq D^{d^k}_{f}(x^k, \hat x)$.
\end{proof}

\begin{remark}
The linear convergence rate from the previous theorem depends on the number of blocks $M$ and shows that considering blocks is beneficial for BK (Algorithm~\ref{alg:BK}). Indeed,  Theorem~\ref{th:linear_convergence_bk} tells us that taking $M = m$ gives us the slowest convergence rate and thus one should consider updating the blocks having a size greater than 1, i.e., $M<m$.  However, small $M$ implies large block sizes which increases the computational time of the subproblem. From our computational results from Section~\ref{sec:numerics} one can see that small $M$ indeed is beneficial in terms of number of full iterations.  Note that Theorem~\ref{th:linear_convergence_bk} generalizes the convergence results from~\cite{P15,LS19}. In fact, to get the linear convergence rate  in~\cite{P15} the authors needed the objective function to be smooth which is not the case in our theorem. The linear convergence result in~\cite{LS19} can be obtained from our Theorem~\ref{th:linear_convergence_bk} by setting $\alpha =1, M=m$. In this case $M \bar L = \|\mathbf{A}\|^2_F$, which corresponds to the result in~\cite[Theorem 3.2]{LS19}.
\end{remark}
A well-known drawback of accelerated first-order methods like Algorithm~\ref{alg:ARBK} is their oscillating behavior that slows them down (see e.g Figure~\ref{fig:500_784} and~\ref{fig:block500_784} below). As a first path towards solving these issues, restart schemes have been proposed in the literature, which have been demonstrated to improve convergence in practice by removing the oscillatory behavior \cite{necoara2019linear}. Practically speaking in a restart scheme, Algorithm~\ref{alg:ARBK} is stopped when a certain criterion is met and then restarted using the last value provided by the method as the new initial condition. In the next section, we present a restart scheme for ARBK (Algorithm~\ref{alg:ARBK}) that exhibits linear convergence for our problem~\eqref{eq:PB}.

\subsection{Restarted ARBK}
It has been shown that under the quadratic growth property accelerated gradient
methods can have improved complexity bounds in combination with an appropriate restarting procedure, see e.g.~\cite{fercoq2020restarting,necoara2019linear,necoara2022linear}. Our goal is to design a restarted ARBK akin to the restart strategies in~\cite{fercoq2020restarting,necoara2022linear} with similar properties based on the results in the previous sections. Having defined a set of integers $\{K_0, K_1,...\}$ of frequencies with which one wishes to restart the method, we can write the restarted version of ARBK. In the following version of ARBK in Algorithm~\ref{alg:RARBK1} we restart with these frequencies, but only update the dual variable if we have decreased the objective during the last restart period.

\begin{algorithm}[ht]
  \caption{Restarted-ARBK (RARBK)}
  \label{alg:RARBK1}
  \begin{algorithmic}[1]
    \STATE {Choose  $y^0 \in \RR^m$ and set $\Tilde{y}^0 = y^0$.} 
    \STATE {choose restart periods $\{K_0, \dots, K_r, \dots\} \subseteq \mathbb{N}$.}
    \STATE initialize $r=0$
    \REPEAT
    \STATE $\Bar{y}^{r+1} = \text{ARBK}(f, \Tilde{y}^{r}, K_r)$
    \vspace{0.2cm}
    \STATE update $
  \Tilde{y}^{r+1} =
  \begin{cases}
    \Bar{y}^{r+1}, & \text{if } \Psi(\Bar{y}^{r+1}) \leq \Psi(\Tilde{y}^{r}) \\\\
    \Tilde{y}^{r}, & \text{if } \Psi(\Bar{y}^{r+1}) > \Psi(\Tilde{y}^{r})
  \end{cases}
 $ 
    \vspace{0.2cm}
    \STATE increment $r =  r+1$
    \UNTIL{a stopping criterion is satisfied}
    \RETURN $\Tilde{x}^{r+1} = \nabla f^*(\mathbf{A}^{\top}\Tilde{y}^{r+1})$.
  \end{algorithmic}
\end{algorithm}

The next lemma gives an estimate of the progress we make in the primal variable in one restart period.
\begin{lemma}[Conditional restarting at $y^k$] \label{lm:cond_rest} Let $(y^k, x^k)$ be the iterates of ARBK (Algorithm~\ref{alg:ARBK}). Let 
\begin{align}
\label{eq:cond_restart}
  \Tilde{y} & = 
  \begin{cases}
    y^k, & \text{if } \Psi(y^k) \leq \Psi(y^0) \\
    y^0, & \text{if } \Psi(y^k) > \Psi(y^0)
  \end{cases}
\end{align}
and set $\Tilde{d} = \mathbf{A}^{\top} \Tilde{y}$, $\Tilde{x} = \nabla f^{*}(\Tilde{d})$.
Then, we have
\begin{equation}
\label{eq:cond_restart_inq}
    \mathbb{E}[D^{\Tilde{d}}_{f}(\Tilde{x}, \hat x)] \leq \theta^2_{k-1} \Bigg(\frac{1-\theta_0}{\theta^2_0}  + \frac{L_{\text{max}}}{\gamma(\hat x)\theta^2_0} \Bigg)\cdot D^{d^{0}}_{f}(x^{0}, \hat x),
\end{equation}
with $L_{\text{max}} = \max\limits_{i \in [M]} \|\mathbf{A}_{(i)}\|_2^2, \theta_0 = \frac{1}{M}$. Moreover, given $\zeta <1$, if 
\begin{equation}
\label{eq:nb_of_steps}
    k \geq \frac{2}{\theta_0} \Bigg(\sqrt{\frac{L_{\text{max}}+\gamma(\hat x)}{\zeta \gamma(\hat x)}} -1\Bigg) +1,
\end{equation}
then $\mathbb{E}[D^{\Tilde{d}}_{f}(\Tilde{x}, \hat x)] \leq \zeta \cdot D^{d^{0}}_{f}(x^{0}, \hat x)$.
\end{lemma}
\begin{proof}
Since $\Psi$ satisfies now the PL inequality~\eqref{eq:pl} it also satisfies the quadratic growth condition~\cite[Theorem 2]{karimi2016linear}, i.e.,  
\begin{equation}
     \Psi(y^k) - \hat \Psi \geq  \frac{\gamma(\hat x)}{2} \text{dist}(y^k, \mathcal{Y}^*)^2.
\end{equation}
From Lemma \ref{lm:fercoq}, the following holds for the iterates of ARBK:
\begin{align*}
    \mathbb{E} [\Psi(y^{k}) - \hat \Psi] &\leq \theta^2_{k-1} \Bigg(\frac{1-\theta_0}{\theta^2_0} (\Psi(y^{0}) - \hat \Psi) + \frac{1}{2\theta^2_0} \text{dist}_{\mathbf{B}}(y^0 , \mathcal{Y}*)^2\Bigg) \\
    &\leq \theta^2_{k-1} \Bigg(\frac{1-\theta_0}{\theta^2_0} (\Psi(y^{0}) - \hat \Psi) + \frac{L_{\text{max}}}{2\theta^2_0} \text{dist}(y^0 , \mathcal{Y}*)^2\Bigg) \\
    &\leq \theta^2_{k-1} \Bigg(\frac{1-\theta_0}{\theta^2_0}  + \frac{L_{\text{max}}}{\gamma(\hat x)\theta^2_0} \Bigg)(\Psi(y^{0}) - \hat \Psi),
\end{align*}
where $\text{dist}_{\mathbf{B}}(y^0 , \mathcal{Y}*)$ denotes the distance of $y^0$ to the set $\mathcal{Y}*$ with respect to the norm $\|\cdot\|_{\mathbf{B}}$ with ${\mathbf{B}} = \textbf{Diag}(\|\mathbf{A}_{(1)}\|_2^2, \|\mathbf{A}_{(2)}\|_2^2, \dots, \|\mathbf{A}_{(M)}\|_2^2)$ and $L_{\text{max}} = \max_{i \in [M]} \|\mathbf{A}_{(i)}\|_2^2$. Furthermore from Lemma~\ref{lm:IterateNormSuboptDual}, it holds $D_f^{d^k}(x^k,\hat x) = \Psi(y^k)- \hat \Psi$ and from condition~\eqref{eq:cond_restart}, we have $D^{\Tilde{d}}_{f}(\Tilde{x}, \hat x) \leq D^{d}_{f}(x^k, \hat x)$ which gives Eq.~\eqref{eq:cond_restart_inq}. Condition~\eqref{eq:nb_of_steps} is equivalent to:
\begin{equation*}
    \frac{4}{(k-1+2M)^2} \Bigg(\frac{1-\theta_0}{\theta^2_0}  + \frac{L_{\text{max}}}{\gamma(\hat x)\theta^2_0} \Bigg) \leq \zeta.
\end{equation*}
\end{proof}

\begin{corollary}
\label{th:restart_lin_conv}
    Denote $K(\zeta) = \bigg\lceil{2M \bigg(\sqrt{\frac{L_{\text{max}}+\gamma(\hat x)}{\zeta \gamma(\hat x)}} -1\bigg) +1 \bigg\rceil}$. If the restart periods $\{K_0, \dots , K_r,\dots\}$ are all equal to $K(\zeta)$, then the iterates of Algorithm~\ref{alg:RARBK1} satisfy:
    \begin{equation*}
        \mathbb{E} [D^{\Tilde{d}^r}_{f}(\Tilde{x}^r, \hat x)] \leq \zeta^r \cdot D^{d^{0}}_{f}(x^{0}, \hat x),
    \end{equation*}
    where the iterate $\Tilde{x}^r$ is obtained after restarting $r$ times ARBK  performing each time $K(\zeta)$ iterations. 
\end{corollary}

\begin{proof}
By the definition of $\Tilde{x}^r$, we know that for all $r$, $D^{\Tilde{d}^{r+1}}_{f}(\Tilde{x}^{r+1}, \hat x) \leq \zeta  \cdot D^{\Tilde{d}^{r}}_{f}(\Tilde{x}^{r}, \hat x)$. Applying recursively Lemma~\ref{lm:cond_rest} gives us the desired linear convergence.
\end{proof}
\noindent Note that the choice $K(\zeta)$ from the previous corollary is not known to be optimal, but such a choice guarantees linear convergence. From relation~\eqref{eq:D} and Corollary~\ref{th:restart_lin_conv} we can also derive a linear rate in terms of
the expected quadratic distance of this primal sequence $\Tilde{x}^r$ to the optimal solution $\hat x$
\begin{equation*}
    \mathbb{E}[\|\Tilde{x}^r - \hat x \|_2^2] \leq 2\zeta^r \cdot D^{d^{0}}_{f}(x^{0}, \hat x)
\end{equation*}
The iterate $\Tilde{x}^r$ is obtained after running $r$ times Algorithm~\ref{alg:RARBK1} of $K(\zeta)$ iterations, which suggests $\zeta = e^{-2}$, or equivalently a fixed restart period (see also \cite{necoara2019linear}): 
\begin{equation}
\label{eq:nb_of_steps1}
  K^{*} = \bigg\lceil{2eM \bigg(\sqrt{\frac{L_{\text{max}} +\gamma(\hat x)}{ \gamma(\hat x)}} -1\bigg) +1 \bigg\rceil}  
\end{equation}

\noindent Even though the quantity $\gamma(\hat{x})$ is not known, Corollary~\ref{th:restart_lin_conv} shows that expected linear convergence can be obtained for restarted ARBK with a long enough period $K$. In fact, we can easily see that those restarting periods depend on the number of blocks $M$ and hence block variants, i.e.,  $M < m$, have shorter restart periods, meaning faster convergence. Hence, it is beneficial to work with block of coordinates. Further, we present a restarting variation of ARBK as proposed in~\cite{fercoq2020restarting}, which does not require explicit knowledge of $\gamma(\hat{x})$. Finally, we fix the length of the first epoch, $K_0$, to
\begin{equation*}
  K_0 = \bigg\lceil{2eM \bigg(\sqrt{\frac{L_{\text{max}} +\bar \gamma(\hat x)}{\bar  \gamma(\hat x)}} -1\bigg) +1 \bigg\rceil}  
\end{equation*}
where $\bar \gamma(\hat x)$ is an estimate  of the unknown constant $\gamma(\hat x)$,  which cannot be chosen arbitrarily. The true value of $\gamma(\hat x)$ can be estimated by some $\bar \gamma(\hat x)$ in Theorem~\ref{th:error_bounds_equality} using Eq.~\eqref{eq:EB}.

\begin{theorem}
\label{th:linear_convergence_restart}
    Let the sequence $\{K_j\}_{j\geq0} \subset \mathbb{N}$ satisfy:
    \begin{enumerate}
        \item $K_{2^j - 1} = 2^j K_0$ \quad for all $j \in \mathbb{N}$
        \item $|\{0 \leq r < 2^p -1|\,\,K_r = 2^jK_0\}| = 2^{p-1-j}$ \quad for all $j \in \{0, 1, \dots, p-1\}$.
    \end{enumerate}
For example, we can consider, $K_0 = K_0, K_1=2^1K_0, K_2=K_0, K_3=2^2K_0$ and so on. Then, after $p$ epochs of Algorithm~\ref{alg:RARBK1} we have the following linear rate in expectation:
\begin{equation*}
        \mathbb{E} [D^{\Tilde{d}^k}_{f}(\Tilde{x}^k, \hat x)] \leq \bigg(e^{- \frac{4}{(p+2)2^{i}K_0}}\bigg)^k \cdot D^{d^{0}}_{f}(x^{0}, \hat x),
    \end{equation*}
with $i = \lceil \max(0, \text{log}_2(K^{*}/K_0)) \rceil$ and $k = \sum \limits_{j=0}^{2^p-1} K_j$ is the total number of iterations.
\end{theorem}

\begin{proof}
    Let us define 
     $\Delta_r = \mathbb{E} [D^{\Tilde{d}^r}_{f}(\Tilde{x}^r, \hat x)]$
     and
    \begin{equation*}
        c_i(p) = |\{l < 2^p -1|\,\,K_l \geq 2^i K_0\}| + 1 = 1 + \sum \limits_{k=i}^{p-1} 2 ^{p-1-k} = 2 ^{p-i}.
    \end{equation*}
    Note that $c_i(p)$ represents the number of restarts such that $K_l \geq K^{*}$, i.e. the number of restarts for
which Lemma~\ref{lm:cond_rest} applies. When
$K_l < K^{*}$, we invoke the guaranteed Bregman distance decrease enforced in Algorithm~\ref{alg:RARBK1}. Then it follows from Lemma~\ref{lm:cond_rest} that
\begin{equation}
\label{eq:variant_restart}
        \Delta_{2^p-1} \leq e^{-2c_i(p)} \cdot \Delta_{0}.
\end{equation}
Moreover, we have :
\begin{align*}
    k = \sum \limits_{j=0}^{2^p-1} K_j &= \sum \limits_{j=0}^{p-1} |\{0 \leq r < 2^p -1|\,\,K_r = 2^jK_0\}| \cdot 2^j K_0 + K_{2^p -1} \\
    &= \sum \limits_{j=0}^{p-1} 2^{p-1-j} \cdot 2^j K_0 + K_{2^p -1} \\
    &= (p+2)2^{p-1}K_0
\end{align*}
and the relation~\eqref{eq:variant_restart} implies:
\begin{equation*}
    \Delta_{k} \leq \bigg(e^{-2 \frac{c_i(p)}{\sum_{j=0}^{2^p-1} K_j}} \bigg)^k \Delta_{0} = \bigg(e^{- \frac{2^{p-i+1}}{(p+2)2^{p-1}K_0}}\bigg)^k \Delta_{0}
\end{equation*}
which confirms the above result.
\end{proof}

\begin{remark}
    \begin{enumerate}
        \item The linear convergence rate stated in Theorem~\ref{th:linear_convergence_bk} clearly implies the following estimate on the total number of iterations required by BK to obtain an $\varepsilon$-suboptimal solution in expectation:
        \begin{equation*}
            \mathcal{O}\bigg(2M \gamma(\hat x) \bar L_{\alpha} \bar{L}^{1-\alpha} \text{log}\bigg(\frac{1}{\varepsilon}\bigg)\bigg).
        \end{equation*}

        \item It follows from Theorem~\ref{th:linear_convergence_restart} that an upper bound on the total number of iterations performed by the Restarted-ARBK scheme to attain an $\varepsilon$-suboptimal solution in expectation is given by:
        \begin{equation*}
            \mathcal{O}\bigg( M \sqrt{\frac{L_{\text{max}}}{\gamma(\hat x)}} \text{log}\bigg(\frac{1}{\varepsilon}\bigg) \text{log}_2 \bigg(\text{log}\bigg(\frac{f(\hat x)}{\varepsilon}\bigg)\sqrt{\frac{\bar \gamma(\hat x)}{\gamma(\hat x)}}\bigg)\bigg).
        \end{equation*}

        \item If we assume that $f(\hat x) \leq 1$ and $\gamma(\hat x)$ is known, then the Restarted-ARBK is better than BK when $\text{log}_2 (\text{log}(1/\varepsilon)) \leq 2\gamma \sqrt{L_{\text{max}}\gamma}$.  For instance, let's suppose that the rows of the matrix are normalized and set $M=m$, then $L_{\text{max}}=1$. If $\gamma = 10^{1/6}$, then Restarted-ARBK has a better worst-case complexity than ARBK for all accuracies $\varepsilon > 10^{-6}$, $\gamma = 10^{1/3}$ for $\varepsilon > 10^{-35}$ and $\gamma = 10^{2/3}$ for $\varepsilon > 10^{-455391}$.
    \end{enumerate}
\end{remark}

\section{Experiments}
\label{sec:numerics}
In this section, we study the computational behaviour of the proposed algorithms, BK (Algorithm~\ref{alg:BK}), ARBK (Algorithm~\ref{alg:ARBK}), and RARBK (Algorithm~\ref{alg:RARBK1}), for the application of finding sparse solutions of linear systems. The first part shows a visual illustration on a simple $\mathbb{R}^2$ example.

\subsection{Visualizing the acceleration mechanism}
This section shows the graphical illustration of our acceleration mechanism. Our goal is to shed more light on how the proposed algorithm works in practice. For simplicity, we illustrate this by comparing BK and ARBK.

\begin{figure}[htb]
  \centering
  \includegraphics[width=0.45\textwidth]{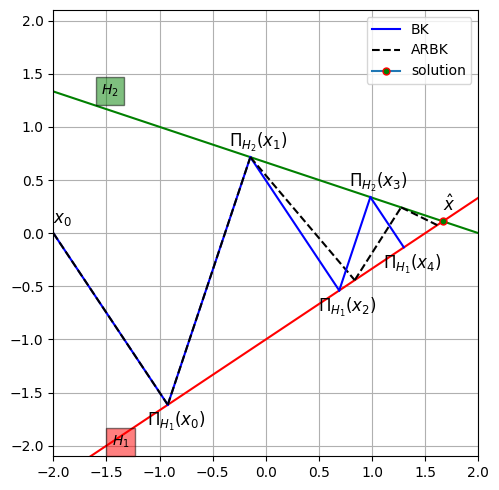}
  \caption{Graphical interpretation of the randomized Bregman Kaczmarz method (BK) and the accelerated randomized Bregman Kaczmarz method (ARBK) using our scheme in a simple example with only two hyperplanes $H_i = \{x: \langle a_i, x \rangle = b_i \}$, with $i=1,2$, and a unique solution $\hat x$. }
  \label{fig:kaczmarz-interpretation}
\end{figure}

\noindent In Fig.~\ref{fig:kaczmarz-interpretation} we present a simple $\mathbb{R}^2$ illustration of the difference between the workings of BK and ARBK with $f(x) = \frac{1}{2}\|x\|^2_2$. Our goal is to show graphically how the addition of interpolation leads to acceleration. In the example of Fig.~\ref{fig:kaczmarz-interpretation}, the performance of ARBK is similar to the performance of BK until iterate $x_3$. After this point, the interpolation parameter becomes more effective and the ARBK method accelerates. This behavior appears also in our experiments in the next section where we work with matrices with many rows. There we can notice that the interpolation parameter seems to become more effective after the first $m+1$ iterations, where $m$ is the number of hyperplanes. For the BK method, it is known that step 6 in Algorithm~\ref{alg:BK} corresponds to the orthogonal projections of the iterates onto one of the hyperplanes. The ARBK method is also doing orthogonal projections initially and after some iterations starts doing oblique projections.

\subsection{Numerical Experiments}
\noindent We present several experiments to demonstrate the effectiveness of Algorithms~\ref{alg:BK}, \newline ~\ref{alg:ARBK} and~\ref{alg:RARBK1} under various conditions\footnote{The code that produces the figures in this paper is available at \url{https://github.com/tondji/arbk}.}. The simulations were performed in \texttt{Python} on an Intel Core i7 computer with 16GB RAM. For all the experiments we consider $f(x) = \lambda \cdot \|x\|_1 + \frac{1}{2}\|x\|^2_2$, where $\lambda$ is the sparsity parameter.

\subsubsection{Synthetic experiments}
In the first part, synthetic data for the experiments is generated as follows:
all elements of the data matrix $\mathbf{A} \in \RR^{m\times n}$ are chosen independent and identically distributed from the standard normal distribution $\mathcal{N}(0, 1)$. We constructed overdetermined, square, and underdetermined linear systems. To construct sparse solutions $\hat x \in \RR^n$, we generate a random vector $y$ from the standard normal distribution $\mathcal{N}(0, 1)$ and we set $\hat x = S_{\lambda}(\mathbf{A}^\top y)$, which comes from Eq.~\eqref{eq:opt_sol} and the corresponding right hand side is $b = \mathbf{A}\hat x \in \RR^m$. For each experiment, we run independent trials each starting with the initial iterate $x_0=0$ for the BK method, $y_0=0$ for other methods and the number of blocks $M$ will always be specified. We measure performance by plotting the relative residual error $\| \mathbf{A}x_k - b\|_2/\|b\|_2$ and the relative error $\|x_k - \hat x \|_2/\|\hat x\|_2$ against the number of epochs (number of passing through all rows of $\mathbf{A}$). In the tables, we report the time when one of the relative errors or relative residuals is below the given tolerance. All the experiments are run for a maximum number of epochs in case the tolerance \texttt{tol} is not reached. During every experiment, the $i$-th row block of the matrix $\mathbf{A}$ is sampled with probabilities according to Eq.~\eqref{eq:proba}. We compared the four following methods for a maximum of  $200*\max(m,n)$ full iterations (epochs):
\begin{enumerate}
    \item \textbf{BK}: The randomized Bregman-Kaczmarz method, i.e., Algorithm~\ref{alg:BK}, with $\lambda=15$ and $M \in \{m/4, m/2\}$. 
    \item \textbf{NRBK}: The Nesterov randomized Bregman-Kaczmarz method~\cite[Method ACDM in Section 5]{nesterov2012efficiency}, with $\lambda=15$ and $M \in \{m/4, m/2\}$. 
    \item \textbf{ARBK}: The accelerated randomized Bregman-Kaczmarz method, i.e., Algorithm~\ref{alg:ARBK}, with $\lambda=15$ and $M \in \{m/4, m/2\}$.
    \item \textbf{RARBK}: The restarted accelerated randomized Bregman-Kaczmarz method, i.e., Algorithm~\ref{alg:RARBK1}, with $\lambda=15$, $M \in \{m/4, m/2\}$ with restarting periods $K^* \in \{165M, 200M\}$.
\end{enumerate}
\noindent In Figure~\ref{fig:500_784} and Figure~\ref{fig:700_700}, we report relative errors and residuals of the Bregman Kaczmarz method (BK), the accelerated Bregman Kaczmarz method (ARBK), it restarted version (RARBK) and the Nesterov acceleration (NRBK) with block size equal to $M =m/4$ and $M=m/2$ for undertermined and squared matrices,  respectively. For the RARBK method, we chose as restart period $K^*=165M$ in Figure~\ref{fig:500_784} and $K^*=200M$ in Figure~\ref{fig:700_700}. Two phases can be observed for the RARBK method. At the beginning, we observe similar convergence to the ARBK method, where restarting does not help. Then acceleration kicks in after a nonnegligible time in fact after the first restarting period.
We observe that the ARBK method accelerates the BK method and the restarted version of ARBK, namely RARBK, gives us the fastest convergence, thus accelerating the ARBK method. From Table~\ref{tab:cpu_time500_784} and Table~\ref{tab:cpu_time700_700}, we also see that the RARBK method needs the least computational time to achieve a relative error of $10^{-6}$, followed by the ARBK method (in fact, RARBK is about 4 times faster than BK). For both the BK and ARBK method we observed a plateau after some epochs. Although both methods in theory should have their errors going towards zero, in practice the convergence is affected by the sparsity parameter $\lambda$,  which in this case is large. However,  if we run both methods for longer times we will see eventually their errors converging towards zero.

\medskip

\noindent In Figure~\ref{fig:block500_784} and Figure~\ref{fig:rblock500_784}, we analyze the effects of the number of blocks $M$ on the relative error and residual for Algorithm \ref{alg:ARBK} and Algorithm \ref{alg:RARBK1}. We used an underdetermined and consistent system, with $M \in \{m, m/2, m/4, m/10, m/100\}$, $\lambda = 15$ and we choose as restarting period $K^*=165M$. As the number of blocks $M$ decreases, i.e., increasing the number of rows used in one iteration, we see a corresponding decrease in the number of iterations needed to reach a certain accuracy, however at some point decreasing $M$ does not improve the method. Moreover, we observed that the RARBK compared to ARBK always needs fewer iterations to reach the same accuracy. These observations are in line with Table~\ref{tab:cpu_timeblock500_784} and Table~\ref{tab:cpu_timerblock500_784} in terms of computational times. A big dip in the relative error at the few last epochs is observed. This behaviour seems to appear also on numerical experiments for different types of problems in other works (see e.g.,  Fig 2 in ~\cite{fercoq2020restarting} and Fig 1 in ~\cite{necoara2022linear}). Our theory only shows linear convergence.

\begin{figure}[htb]%
    \centering
    \subfloat[\centering Relative Residual]{{\includegraphics[width=.45\linewidth]{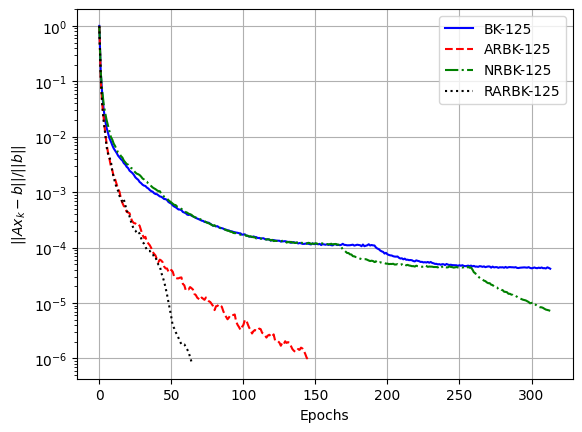} }}%
    \qquad
    \subfloat[\centering Relative Error]{{\includegraphics[width=.45\linewidth]{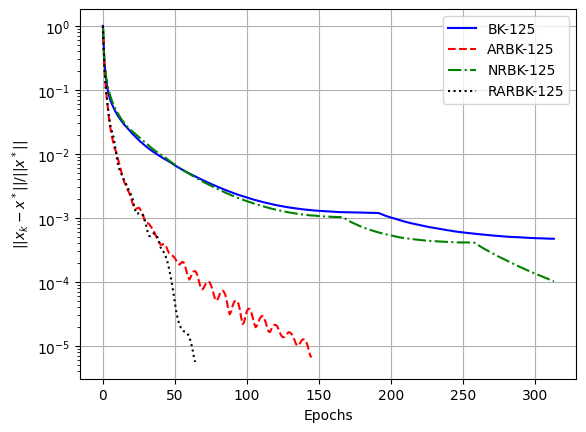} }}
    \caption{\small A comparison of Bregman Kaczmarz (BK-blue), accelerated randomized Bregman Kaczmarz (ARBK-red), Nesterov Acceleration scheme (NRBK-green) and restarted version of ARBK (RARBK-black), $m = 500, n = 784$, sparsity $s=408$, $M=m/4$, $\lambda=15$, $\kappa(A) =8.98$, \texttt{tol}=$10^{-6}$, $K^* = 165M$.}%
    \label{fig:500_784}%
\end{figure}

\begin{table}[htb]
\centering
\begin{tabular}{lr}
\toprule
 & CPU time (s) \\ \midrule
BK-125             &  46.58 $^*$                          \\
ARBK-125           &  24.65                                \\
NRBK-125           &  53.16 $^*$                           \\
RARBK-125          &  \textbf{11.86}                        \\ \bottomrule
\end{tabular}
\caption{CPU time until the relative error or residual falls below \texttt{tol} for Figure~\ref{fig:500_784}. Here $"*"$ indicates that the \texttt{tol}=$10^{-6}$ has not been reached by neither the relative error nor the relative residual after the maximum number of epochs.}
\label{tab:cpu_time500_784}
\end{table}

\begin{figure}[htb]%
    \centering
    \subfloat[\centering Relative Residual]{{\includegraphics[width=.45\linewidth]{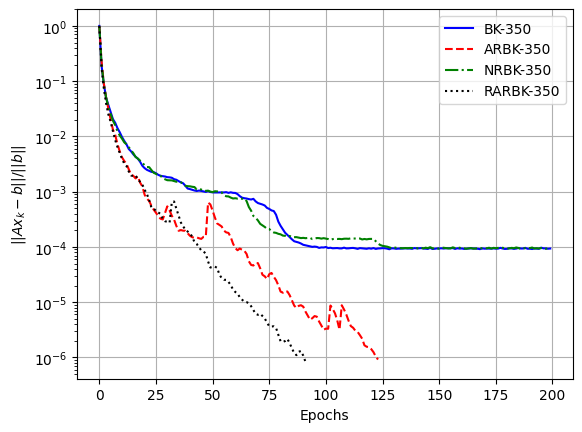} }}%
    \qquad
    \subfloat[\centering Relative Error]{{\includegraphics[width=.45\linewidth]{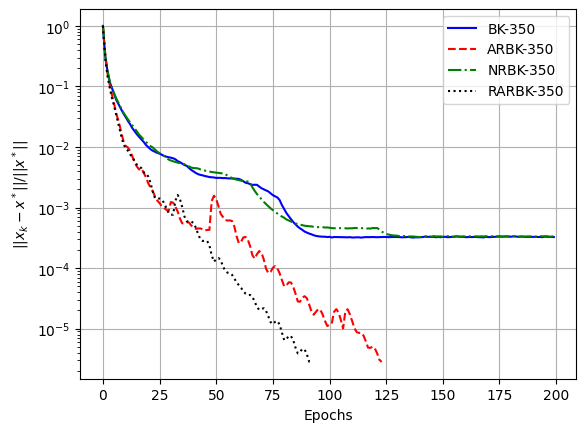} }}
    \caption{\small A comparison of Bregman Kaczmarz (BK-blue), accelerated randomized Bregman Kaczmarz (ARBK-red), the Nesterov Acceleration scheme (NRBK-green) and the restarted version of ARBK (RARBK-black), $m = 700, n = 700$, sparsity $s=404$, $M=m/2$, $\lambda=15$, $\kappa(A) =1150.06$, \texttt{tol}=$10^{-6}$, $K^*=200M$.}%
    \label{fig:700_700}%
\end{figure}

\begin{table}[htb]
\centering
\begin{tabular}{lr}
\toprule
 & CPU time (s)\\ \midrule
BK-350             & 56.70 $^*$                        \\
ARBK-350          & 35.21                         \\
NRBK-350           & 62.13 $^*$                        \\
RARBK-350          & \textbf{27.38}                         \\ \bottomrule
\end{tabular}
\caption{CPU time until the relative error or residual falls below \texttt{tol} for Figure~\ref{fig:700_700}. Here $"*"$ indicates that the \texttt{tol}=$10^{-6}$ has not been reached by neither the relative error nor the relative residual after the maximum number of epochs.}
\label{tab:cpu_time700_700}
\end{table}

\begin{figure}[htb]%
    \centering
    \subfloat[\centering Relative Residual]{{\includegraphics[width=.45\linewidth]{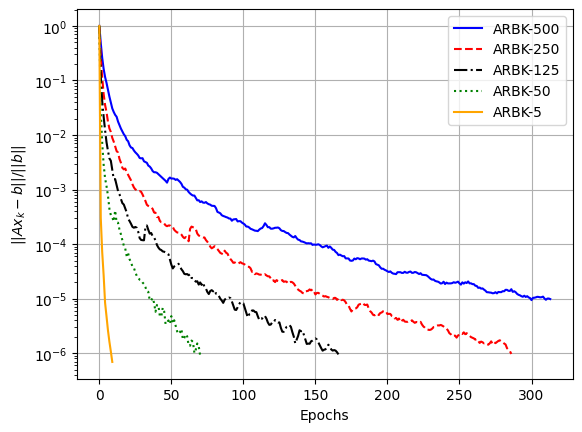} }}%
    \qquad
    \subfloat[\centering Relative Error]{{\includegraphics[width=.45\linewidth]{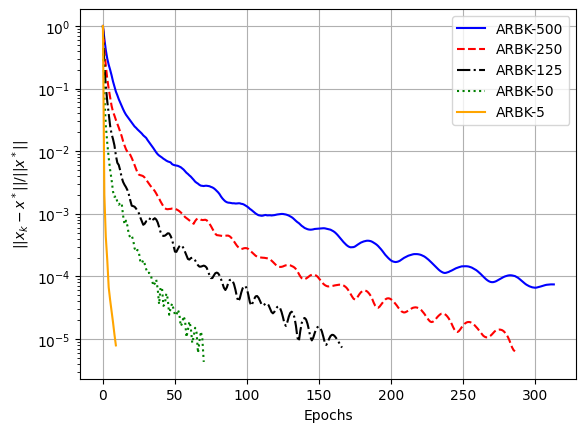} }}
    \caption{A comparison of the accelerated randomized Bregman Kaczmarz (ARBK) for different numbers of blocks $M$, with $m = 500, n = 784$, sparsity $s=408$, $\lambda=15$, $\kappa(A) =8.98$, \texttt{tol}=$10^{-6}$, $K^* = 165M$.}%
    \label{fig:block500_784}%
\end{figure}

\begin{table}[htb]
\centering
\begin{tabular}{lr}
\toprule
 & CPU time (s)\\ \midrule
ARBK-500            &     65.28 $^*$                        \\
ARBK-250           & 51.82                      \\
ARBK-125           & 31.08                         \\
ARBK-50          & 11.30                         \\ 
ARBK-5          & \textbf{2.01} \\\bottomrule
\end{tabular}
\caption{CPU time until the relative error or residual falls below \texttt{tol} for Figure~\ref{fig:block500_784}. Here $"*"$ indicates that the \texttt{tol}=$10^{-6}$ has not been reached by neither the relative error nor the relative residual after the maximum number of epochs.}
\label{tab:cpu_timeblock500_784}
\end{table}

\begin{figure}[htb]%
    \centering
    \subfloat[\centering Relative Residual]{{\includegraphics[width=.45\linewidth]{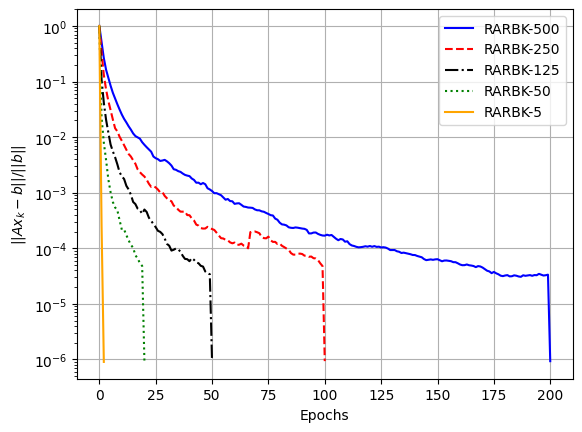} }}%
    \qquad
    \subfloat[\centering Relative Error]{{\includegraphics[width=.45\linewidth]{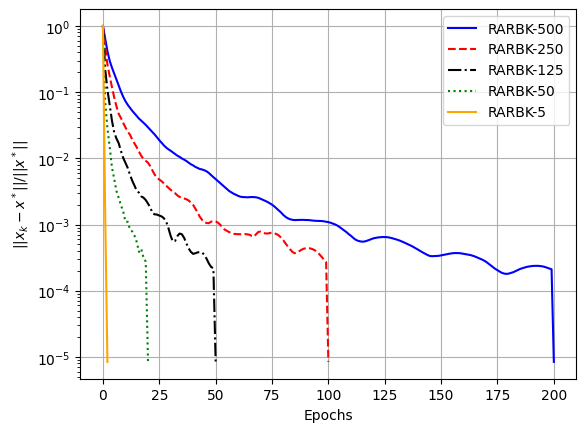} }}
    \caption{\small A comparison of restart accelerated randomized Bregman Kaczmarz (RARBK) for different numbers of blocks $M$, with $m = 500, n = 784$, sparsity $s=408$, $\lambda=15$, $\kappa(A) =8.98$, \texttt{tol}=$10^{-6}$, $K^* = 165M$.}%
    \label{fig:rblock500_784}%
\end{figure}

\begin{table}[htb]
\centering
\begin{tabular}{lr}
\toprule
 & CPU time (s)\\ \midrule
RARBK-500             & 41.94                         \\
RARBK-250           & 18.25                         \\
RARBK-125           & 7.87                         \\
RARBK-50          & 2.93                         \\ 
RARBK-5          & \textbf{0.53} \\\bottomrule
\end{tabular}
\caption{CPU time until the relative error or residual falls below \texttt{tol} for Figure~\ref{fig:rblock500_784}.}
\label{tab:cpu_timerblock500_784}
\end{table}

\subsubsection{Computerized tomography}
In this section, we consider the application of computerized tomography (CT), where the Radon transform implementation from the \texttt{Python} library \texttt{skimage} was used to build a system matrix for a parallel beam CT for a phantom of size $N\times N$ with $N=50$ and with $60$ equispaced angles. Hence, the system matrix $\mathbf{A}$ has size $3000\times 2500$, i.e., $m=3000$ and $n=2500$. We interpret the projection for each angle as one block $\mathbf{A}_{(i)}$, i.e., we have $M=60$ and each block has the size $m_{i}=50$. The ground truth solution $\hat x$ is fairly sparse (see below) and we generated the exact right-hand side simply as $b=\mathbf{A} \hat x$.
We used different methods for reconstruction, each run for $10$ epochs using the function $f(x) = \lambda \cdot \|x\|_1 + \frac{1}{2}\|x\|^2_2$:
\begin{enumerate}
    \item \textbf{BK}: Randomized Bregman-Kaczmarz method, i.e., Algorithm~\ref{alg:BK}, with $\lambda=30$ and $M=60$.

    \item \textbf{NRBK}: Nesterov randomized Bregman-Kaczmarz method~\cite[Method ACDM in Section 5]{nesterov2012efficiency}, with $\lambda=30$ and $M=60$.

    \item \textbf{ARBK}: Accelerated randomized Bregman-Kaczmarz method, i.e., Algorithm~\ref{alg:ARBK}, with $\lambda=30$ and $M=60$.

    \item \textbf{RARBK}: Restarted accelerated randomized Bregman-Kaczmarz method, i.e., Algorithm~\ref{alg:RARBK1}, with $\lambda=30$, $M=60$ and restarting periods $K^*=165M$.
\end{enumerate}
In this part, the BK algorithm represents our baseline comparison. It can be applied by just choosing $\lambda$ by trial and error to adapt the sparsity of the outcome. Figure~\ref{fig:ct_reconstruction} shows the ground truth data and the reconstructions by BK, NRBK, ARBK and RARBK. As can be seen, all methods produce a clear background due to the sparsity enforcing $f$. Both ARBK and RARBK methods lead to better reconstruction. Table~\ref{tab:cpu_time_ct} shows the CPU time for the four methods in order to reach a relative error or relative residual below \texttt{tol}. We notice again that RARBK is the fastest method (about 2.5 times faster than BK).

\begin{figure}[htb]
  \centering
  \includegraphics[width=1\textwidth]{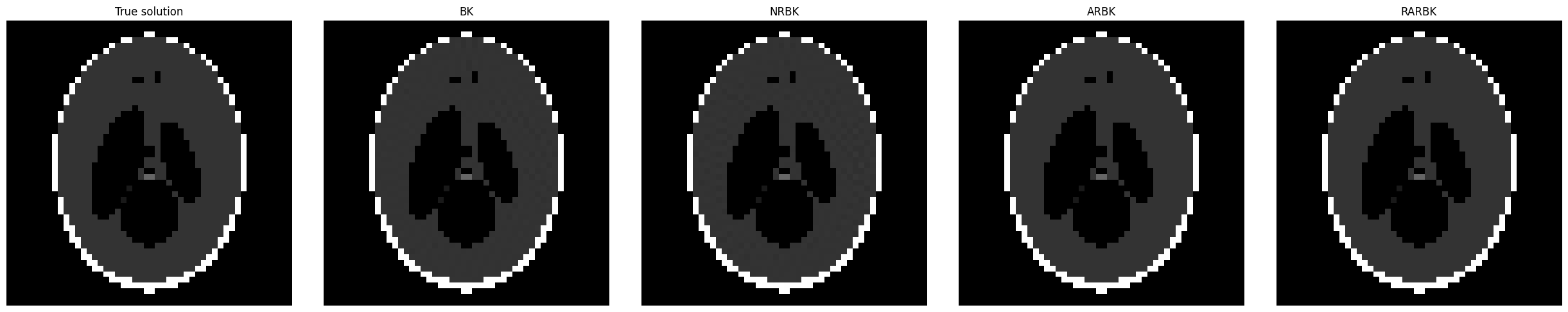}
\caption{Ground truth solution and reconstruction by the different methods in Figure~\ref{fig:ct_example}.}
  \label{fig:ct_reconstruction}
\end{figure}

\begin{table}[htb]
\centering
\begin{tabular}{lr}
\toprule
 & CPU time (s) \\ \midrule
BK-60             & 202.17 $^*$                         \\
ARBK-60          & 165.35                         \\
NRBK-60           & 203.63 $^*$                         \\
RARBK-60          & \textbf{84.84}                         \\ \bottomrule
\end{tabular}
\caption{CPU time until the relative error or residual falls below \texttt{tol} for Figure~\ref{fig:ct_example}. Here $"*"$ indicates that the \texttt{tol}=$10^{-5}$ has not been reached by neither the relative error nor the relative residual after the total number of epochs.}
\label{tab:cpu_time_ct}
\end{table} 

\noindent In Figure~\ref{fig:ct_example}, we show the relative residuals and the relative reconstruction errors for all methods. While the BK and NRBK methods bring the residual down to some value, the accelerated method (ARBK) and its restarted version (RARBK ) are able to go way below that and it seems that they even go down further. It is important to note that the choice of the restarting periods $K^*$ is crucial for making the RARBK work properly and accelerating the ARBK method. The parameter $K^*$ is difficult to estimate (see Eq.~\eqref{eq:nb_of_steps1}) as it is related to the quadratic growth of the dual objective $\gamma(\hat x)$. Nevertheless a restarting strategy with long enough periods $K^*$ leads to good results in most experiments.

\begin{figure}[htb]%
    \centering
    \subfloat[\centering Relative Residual]{{\includegraphics[width=.45\linewidth]{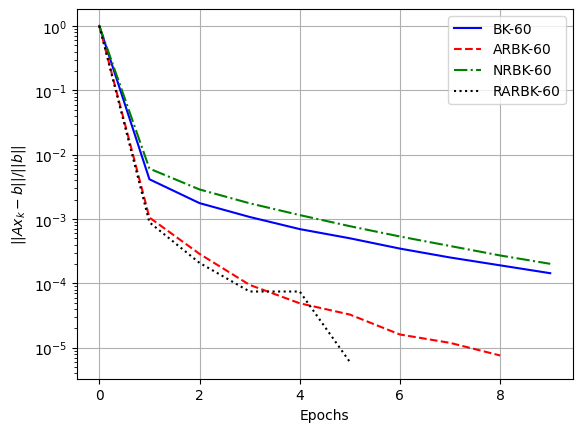} }}%
    \qquad
    \subfloat[\centering Relative Error]{{\includegraphics[width=.45\linewidth]{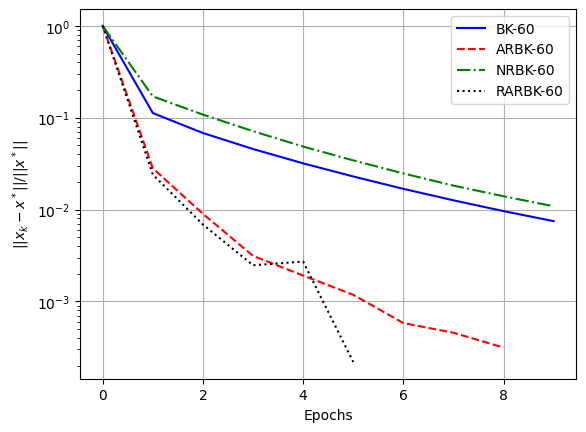} }}
    \caption{\small A comparison of Bregman Kaczmarz (BK-blue), accelerated randomized Bregman Kaczmarz (ARBK-red),  Nesterov Acceleration scheme (NRBK-green) and  restarted version of ARBK (RARBK-black), $m = 3000, n = 2500$, sparsity $s=912$, $M=60$, $\lambda=30$, $\kappa(A)=5411.08$, \texttt{tol}=$10^{-5}$, $K=165M$.}%
    \label{fig:ct_example}%
\end{figure}

\section{Conclusions}
\label{sec:conclusion}
In this work, we have proposed block (accelerated) randomized Bregman-Kaczmarz methods for solving convex linearly constrained optimization problems. Our methods are of Kaczmarz-type, i.e., they use a block of constraints in each iteration. We have used a recently proposed acceleration variant for the randomized coordinate descent method and transferred it to the primal space. A theoretical analysis of the convergence of the proposed nonaccelerated and accelerated algorithms was given. Numerical experiments have shown that the proposed methods are more efficient and faster than the existing methods for solving the same problem.

\bibliography{arbk_new}
\bibliographystyle{plain}

\end{document}